\documentclass[a4paper, 11pt, reqno]{amsart}
\usepackage{amsmath,amsfonts,amssymb,amsthm,enumerate}
\usepackage{graphicx}
\usepackage{color}
\usepackage{xy} \xyoption{all}

\newtheorem{thm}{Theorem}[section]
\newtheorem{cor}[thm]{Corollary}
\newtheorem{prop}[thm]{Proposition}
\newtheorem{lem}[thm]{Lemma}

\theoremstyle{definition}
\newtheorem{defn}[thm]{Definition}
\newtheorem{exas}[thm]{Example}
\newtheorem{rem}[thm]{Remark}
\newtheorem{conj}[thm]{Conjecture}

\let\phi\varphi

\newcommand{\mg}{\mathcal{G}}

\newcommand{\ds}{\displaystyle}
\usepackage[english]{babel}

\newcommand{\bb}[1]{\mathbb{#1}}
\addto\captionsenglish{}
\newcommand{\vp}{\varphi}

\begin{document}
\title{On the ideals of ultragraph Leavitt path algebras}
\maketitle
\begin{center}
T. T. H.~Duyen\footnote{Faculty of Pedagogy, Hanoi Pedagogical University 2, Hanoi, Vietnam. E-mail address: \texttt{haiduyenclhd@gmail.com}},
D.~Gon\c{c}alves\footnote{Departamento de Matem\'{a}tica, Universidade Federal de Santa Catarina, Florian\'{o}polis, 88040-900, Brazil. E-mail address: \texttt{daemig@gmail.com}}	
and  T.\,G.~Nam\footnote{Institute of Mathematics, VAST, 18 Hoang Quoc Viet, Cau Giay, Hanoi, Vietnam. E-mail address: \texttt{tgnam@math.ac.vn}
		
%\ \ {\bf Acknowledgements}:   The first author was partially supported by the Vietnam Academy of Science and Technology grant CT0000.02/20-21. The authors take an opportunity to express their deep gratitude to the anonymous referee for extremely careful reading, highly professional working with our manuscript, and valuable suggestions.
}
%The author  expresses their deep gratitude to Professor Gene Abrams (Department of Mathematics, University of Colorado, Colorado Springs, Coloralo, USA) for his valuable suggestions which led to the final shape of the paper.} 
	\end{center}
	
\begin{abstract} In this article, we provide an explicit description of a set of generators for any ideal of an ultragraph Leavitt path algebra. We provide several additional consequences of this description, including information about generating sets for graded ideals, the graded uniqueness and Cuntz-Krieger theorems, the semiprimeness, and the semiprimitivity of ultragraph Leavitt path algebras, a complete characterization of the prime and primitive ideals of an ultragraph Leavitt path algebra. We also show that every primitive ideal of an ultragraph Leavitt path algebra is exactly the annihilator of a Chen simple module. Consequently, we prove Exel's Effros-Hahn conjecture on primitive ideals in the ultragraph Leavitt path algebra setting (a conclusion that is also new in the context of Leavitt path algebras of graphs).
\medskip

\textbf{Mathematics Subject Classifications 2020}: 16S88, 16S99,  05C25
		
\textbf{Key words}: Ultragraph Leavitt path algebras, prime and primitive ideals, simple modules, Steinberg algebras, Effros-Hahn conjecture. 
\end{abstract}
	
\section{Introduction}
The study of algebras associated with combinatorial objects is a thriving topic in classical ring theory. One of the key goals of the subject is to establish relationships between combinatorial properties of the initial object and algebraic properties of the associated algebra. Another important direction is the study of the connections with other branches of mathematics, as C*-algebras and symbolic dynamics. Among interesting examples of algebras associated with combinatorial objects we mention, for example, the following ones: graph $C^*$-algebras, Leavitt path algebras, higher rank graph algebras, Kumjian-Pask algebras, and ultragraph $C^*$-algebras (we refer the reader to \cite{a:lpatfd} and \cite{AAS} for a more comprehensive list). 

There is no doubt that, among the non-analytical algebras mentioned above, the Leavitt path algebra associated with a graph figures as the most studied one. For these algebras their structure, and connections with C*-algebra theory and symbolic dynamics, have been (and still is) studied in detail. 

Ultragraphs and ultragraph C*-algebras were defined by Mark Tomforde in \cite{tomf:auatelaacaastg03} as a unifying approach to C*-algebras associated with infinite matrices (also known as Exel-Laca algebras) and graph $C^*$-algebras. They have proved to be a key ingredient in the study of Morita equivalence of Exel-Laca and graph $C^*$-algebras \cite{kmst:gaelaauacutme}. Recently, Castro, Gon\c{c}alves, Royer, Tasca, Wyk, among others, have established nice connections between ultragraph $C^*$-algebras and the symbolic dynamics of shift spaces over infinite alphabets (see \cite{GilDanDan}, \cite{gr:uassoia},  \cite{gr:iaessvuatca} and \cite{tasca}).

The Leavitt path algebra associated with an ultragraph was defined by Imanfar, Pourabbas, and Larki in \cite{ima:tlpaou}. In \cite{gr:saccfulpavpsgrt}, a slightly different definition appeared, and in \cite{dgv:uavlggwatgut} de Castro, Gon\c{c}alves and van Wyk showed that the resulting algebras are isomorphic. As in the C*-algebraic setting, the ultragraph Leavitt path algebras unify the study of Leavitt path algebras associated with graphs and the algebras associated with infinite matrices. Further to being a convenient way to express both types of algebras mentioned, it was shown in \cite{ima:tlpaou} that ultragraph Leavitt path algebras provide examples of algebras that can not be realized as the Leavitt path algebra of a graph; that is, the class of ultragraph Leavitt path algebras is strictly larger than the class of Leavitt path algebras of graphs. In fact, we will build on the work on \cite{ima:tlpaou} and provide a large class of ultragraph algebras that can not be seen as the Leavitt path algebra of a graph. 

Since ultragraph Leavitt path algebras form a strictly larger class than Leavitt path algebras of graphs, their study encompasses an extra layer of complexity. Nevertheless, recently several results regarding whether the $C^*$-algebraic theory of ultragraphs has analogues in the algebraic setting, and whether results about Leavitt path algebras of graphs can be generalized to ultragraph Leavitt path algebras, have been obtained. We mention the following. Gon\c{c}alves and Royer \cite{gr:iaproulpa} extended Chen's construction of irreducible representations of graph Leavitt path algebras to ultragraph Leavitt path algebras (see \cite{c:irolpa});  Gon\c{c}alves and Royer \cite{gr:saccfulpavpsgrt} realized ultragraph  Leavitt path algebras as partial skew group rings. Using this realization they characterized Artinian ultragraph Leavitt path algebras and gave simplicity criteria for these algebras; and  de Castro, Gon\c{c}alves and van Wyk \cite{dgv:uavlggwatgut} realized ultragraph Leavitt path algebras as Steinberg algebras, and applied this result to obtain generalized uniqueness theorems for ultragraph
Leavitt path algebras. Nam and the third author \cite{NN:2020} characterized purely infinite simple ultragraph Leavitt path algebras, and established the Trichotomy Principle for graded simple ultragraph Leavitt path algebras.

The current article is a continuation of this direction. Our goal is to describe many of the ideal related properties of ultragraph Leavitt path algebras. For example, given the important role of prime and primitive ideals in the theory on non-commutative rings we describe when an ultragraph Leavitt path algebra is prime, or primitive, in terms of combinatorial properties of the underlying ultragraph. We also study prime and primitive ideals of ultragraph Leavitt path algebras in detail. As mentioned in \cite{Bell}, in general, it is very difficult to find all irreducible representations of a ring but often just knowing
the annihilators of the simple modules will allow one to prove non-trivial facts about a ring. We, therefore, extend the theory of Chen simple modules to ultragraph Leavitt path algebras and use it to prove the Exel's Effros-Hahn conjecture for ultragraph Leavitt path algebras. In particular, we should mention that our result regarding the Exel's Effros-Hahn conjecture is also new in the context of Leavitt path algebras of graphs.

Many of the results of our paper rely on an explicit description of a set of generators for any ideal of an ultragraph Leavitt path algebra. This is done for Leavitt path algebras of graphs in \cite[Theorem 2.1]{abcr:tscilpaoag} by Abrams, Bell, Colar, and Rangaswamy. The result in \cite[Theorem 2.1]{abcr:tscilpaoag} is the motivational starting point for our work, and to extend this result to ultragraph Leavitt path algebras is the first goal of our paper. Our proof is based on the one of \cite[Theorem 2.1]{abcr:tscilpaoag}, however, we should mention that, although ultragraphs benefit from the same kind of intuition available for graphs, very often the techniques involved in generalizing results know for graphs require non-trivial ideas to be developed. Also, sometimes the graph intuition may be misleading and a detailed approach is necessary to deal with ultragraphs. The reader should keep these points in mind throughout the paper.

%In \cite[Theorem 11]{ap:pislpa06} the authors characterized purely infinite simple Leavitt path algebras in terms of properties of the associated graph. Also, in \cite[Proposition 3.1.14]{AAS} the authors showed that every graded simple  Leavitt path algebra of a graph is either a locally matricial algebra, or a full matrix ring over $K[x, x^{-1}]$, or a purely infinite simple algebra. The authors of \cite{AAS} call this result the Trichotomy Principle for graded simple Leavitt path algebras of graphs. Motivated by these interesting results, the main goal of this article is to characterize purely infinite simple ultragraph Leavitt path algebras, and establish the Trichotomy Principle for graded simple ultragraph Leavitt path algebras. Our proofs are based on the ones of \cite[Theorem 11]{ap:pislpa06} and \cite[Proposition 3.1.14]{AAS}, respectively. However, we should mention that although ultragraphs benefit from the same kind of intuition available for graphs, very often the techniques involved in generalizing results know for graphs require non-trivial ideas to be developed. Also, sometimes the graph intuition may be misleading and a detailed approach is necessary to deal with ultragraphs.

The article is organized as follows. In Section 2, for the reader's convenience, we provide subsequently necessary notions and facts on ultragraphs and ultragraph Leavitt path algebras. We also prove one of our main theorems, that is, we provide an explicit description of a set of generators for any ideal of an ultragraph Leavitt path algebra (Theorem~\ref{generatingset}), as well as an explicit description of a set of generators for any graded ideal of an ultragraph Leavitt path algebra (Theorem~\ref{gi-generatingset}). Moreover, in Proposition~\ref{ExULPaNLPA} we describe a class of ultragraph Leavitt path algebras that are not isomorphic to the Leavitt path algebra of any graph and we show that ultragraph Leavitt path algebras are always semiprime (Theorem~\ref{semiprime}) and always have null Jacobson radical (Theorem~\ref{semiprimitive}). In Section 3, we describe for which ultragraphs the associated Leavitt algebra is prime (Theorem~\ref{prime1}), we characterize the prime ideals of an ultragraph Leavitt path algebra (Theorem~\ref{prime6}), show that every prime ideal of an ultragraph Leavitt path algebra is graded if, and only if, the ultragraph satisfies Condition~(K) (Corollary \ref{prime7}), and provide a method of constructing non-graded prime ideals of an ultragraph Leavitt path algebra (Corollary \ref{prime8}).

We study primitive ideals of ultragraph Leavitt path algebras in Section~4. In particular, we characterize primitive ultragraph Leavitt path algebras $L_K(\mathcal G)$ in terms of the combinatorics of the ultragraph $\mathcal G$ in Theorem~\ref{primitive1}, and we describe the primitive ideals of an ultragraph Leavitt path algebra in Theorem~\ref{primitive3}.
We dedicate Section~5 to the study of Chen simple modules. In the main result of this section (Theorem~\ref{Chenmod4}), we show that for every primitive ideal $P$ of an ultragraph Leavitt path algebra, there exists a Chen simple module $S$ such that the annihilator of $S$ is $P$. Finally, in Section~6, using groupoid theory, we prove the Exel's Effros-Hahn conjecture for ultragraph Leavitt path algebras (Theorem~\ref{Exel's Effros-Hahn}).

\section{Generating sets for ideals in ultragraph Leavitt path algebras}
The main goal of this section is to provide an explicit description of a set of generators for any ideal of an ultragraph Leavitt path algebra (Theorem~\ref{generatingset}). We provide several additional consequences of this description, including information about generating sets for graded ideals (Theorem \ref{gi-generatingset}), the graded uniqueness and Cuntz-Krieger theorems (Theorems \ref{gutheo} and \ref{cgutheo}), the semiprimeness (Theorem \ref{semiprime}) and the semiprimitivity (Theorem \ref{semiprimitive}) of ultragraph Leavitt path algebras.

\subsection{Generating sets for ideals} We begin this subsection by recalling some notions regarding ultragraph theory, as introduced by Tomforde in \cite{tomf:auatelaacaastg03} and \cite{tomf:soua}.

An \textit{ultragraph} $\mathcal{G} = (G^0, \mathcal{G}^1, r, s)$ consists of a countable set of vertices $G^0$, a countable set of edges $\mathcal{G}^1$, and functions $s : \mathcal{G}^1 \longrightarrow G^0$ and $r : \mathcal{G}^1 \longrightarrow \mathcal{P}(G^0)\setminus \left\{ \varnothing\right\}$, where $\mathcal{P}(G^0)$ denotes the set of all subsets of $G^0$.

A vertex $v \in G^0$ is called a \textit{sink} if $s^{-1}(v)=\varnothing$ and $v$ is called an \textit{infinite emitter} if $|s^{-1}(v)|=\infty$. Denote the set of sinks in an ultragraph by $G^0_s$. A \textit{singular vertex} is a vertex that is either a sink or an infinite emitter. The set of all singular vertices is denoted by $\text{Sing}(\mathcal{G})$. 
A vertex $v \in G^0$ is called a \textit{regular vertex} if $0<|s^{-1}(v)|<\infty$. 

For an ultragraph $\mathcal{G} = (G^0, \mathcal{G}^1, r, s)$ we let $\mathcal{G}^0$ denote the smallest subset of $\mathcal{P}(G^0)$ that contains $\{v\}$ for all $v\in G^0$, contains $r(e)$ for all $e\in \mathcal{G}^1$, and is closed under relative complements, finite unions and finite intersections. Elements of $\mathcal{G}^0$ are called {\it generalized vertices}.  

%The following lemma gives us another description of $\mathcal{G}^0$.\begin{lem}[{\cite[Lemma 2.12]{tomf:auatelaacaastg03}}]\label{g^0}	If $\mathcal{G} := (G^0, \mathcal{G}^1, r, s)$ is an ultragraph, then$\mathcal{G}^0 = \{\bigcap_{e\in X_1}r(e)\cup \cdots\cup \bigcap_{e\in X_n}r(e)\cup F\mid X_1, \hdots, X_n \text{ are finite subsets of } \mathcal{G}^1\\ \text{ and } F \text{ is a finite subset of } G^0\}.$ Furthermore, $F$ may be chosen to be disjoint from $\bigcap_{e\in X_1}r(e)\cup \cdots\cup \bigcap_{e\in X_n}r(e)$.\end{lem}

A \textit{finite path} in an ultragraph $\mathcal{G}$ is either an element of $\mathcal{G}^0$, or a sequence $\alpha_1\alpha_2\cdots\alpha_n$ of edges with $s(\alpha_{i+1})\in r(\alpha_i)$ for all $1\le i\le n-1$. We say that the path $\alpha$ has \textit{length} $|\alpha| := n$, consider the elements of $\mathcal{G}^0$ to be paths of length $0$, and denote by $\mathcal{G}^*$ the set of all finite paths in $\mathcal{G}$. The maps $r$ and $s$ extend naturally to $\mathcal{G}^*$. Note that when $A\in \mathcal{G}^0$ we define $s(A) = r(A) = A$. An \textit{infinite path} in $\mathcal{G}$  is a sequence $e_1e_2\cdots e_n\cdots$ of edges in $\mathcal{G}$ such that $s(e_{i+1})\in r(e_i)$ for all $i\geq 1$. We denote by $\mg^{\infty}$ the set of all infinite paths in $\mg$. For $p = e_1e_2\cdots e_n\cdots\in \mg^{\infty}$, we define $s(p) := s(e_1)$ and  $p^0 :=\{v\in G^0\mid v = s(e_i) \text{ for some } i\}$.

If $\mathcal{G}$ is an ultragraph, then a \textit{closed path} in $\mathcal{G}$ is a path $\alpha=\alpha_1\alpha_2\cdots\alpha_{|\alpha|}\in \mathcal{G}^*$ with $|\alpha|\ge 1$ and $s(\alpha)\in r(\alpha)$. We also say that the closed path $\alpha$ is based at $v=s(\alpha)$. A \textit{cycle} (based at $v$) is a closed path (based at $v$) such that $s(\alpha_i)\neq s(\alpha_j)$ for every $1\leq i\neq j\leq |\alpha|$.

%{\color{red} Careful! The above definition, following Tomforde, is not what Gene et all use for graphs. In fact, cycle above should be called a closed path. In our case we want to define a simple closed path, or, in the Abrams et all nomenclature, a cycle. In the main theorem, we are using the notion of cycle, meaning simple closed curve...so we should adjust the definition here, following the one in Abrams et all}

An \textit{exit} for a cycle $\alpha$ is one of the following:
\begin{itemize} 
\item[(1)] an edge $e\in \mathcal{G}^1$ such that there exists an $i$ for which $s(e)\in r(\alpha_i)$ but $e\ne \alpha_{i+1}$.
 
\item[(2)] a sink $w$ such that $w\in r(\alpha_i)$ for some $i$. 
\end{itemize}

% In \cite{tomf:auatelaacaastg03} Mark Tomforde  introduced  the $C^*$-algebra of an ultragraph as an unifying approach to Exel-Laca and graph $C^*$-algebras. Leavitt path algebras of ultragraphs were introduced in the literature (see, e.g., \cite{ima:tlpaou} and \cite{gr:saccfulpavpsgrt}).

\begin{defn}[{\cite[Definition 2.1]{ima:tlpaou}}]\label{utraLevittpathalg}
Let $\mathcal{G}$ be an ultragraph and $K$ a field. The \textit{Leavitt path algebra $L_K(\mathcal{G})$ of $\mathcal{G}$, with coefficients in $K$,} is the $K$-algebra generated by the set $\{s_e, s^*_{e}\mid e\in\mathcal{G}^1\}$ $\cup\left\{p_{_A}\mid A\in\mathcal{G}^0\right\}$, satisfying the following relations, for all $A,B \in \mathcal{G}^0$ and  $e, f\in\mathcal{G}^1$:
\begin{itemize} 	
\item[(1)] $p_{_\emptyset}=0, p_{_A}p_{_B} = p_{_{A\cap B}}$ and $p_{_{A\cup B}} = p_{_A}+p_{_B}-p_{_{A\cap B}}$;
\item[(2)] $p_{s(e)}s_e = s_e = s_ep_{r(e)}$ and $p_{r(e)}s_{e}^* = s_{e}^* = s_{e}^*p_{s(e)}$;
\item[(3)] $s_{e}^*s_f = \delta_{e, f}p_{r(e)}$;
\item[(4)] $p_v = \sum_{s(e)=v}s_es_{e}^*$ for any regular vertex $v$;
\end{itemize} where $p_v$ denotes $p_{_{\{v\}}}$ and $\delta$ is the Kronecker delta.
\end{defn}

It is worth mentioning the following note.

\begin{rem}
There have been different definitions of Leavitt path algebras of ultragraphs, and the difference of these definitions lies in how the set of generalized vertices are defined. Given an ultragraph $\mathcal{G}$, let $\mathcal{B}$ denote the smallest subset of $\mathcal{P}(G^0)$ that contains $\{v\}$ for all $v\in G^0$, contains $r(e)$ for all $e\in \mathcal{G}^1$, and is closed under finite unions and finite intersections. We denote by $L_K(\mathcal{G}_r)$ the Leavitt path algebra associated with $\mathcal{G}$ by allowing $A,B \in \mathcal{B}$ in item (1) of Definition \ref{utraLevittpathalg}, that means,  $L_K(\mathcal{G}_r)$ is the	algebra as defined in \cite[Theorem 2.11]{tomf:auatelaacaastg03} and \cite[Definition 2.3]{gr:saccfulpavpsgrt}. However, in \cite[Proposition 5.2]{dgv:uavlggwatgut} the authors showed that $L_K(\mathcal{G}_r)$ and $L_K(\mathcal{G})$ are isomorphic to each other.
\end{rem}

We usually denote $s_A := p_{_A}$ for $A\in\mathcal{G}^0$ and $s_\alpha := s_{e_1}\cdots s_{e_n}$ for $\alpha = e_1\cdots e_n \in \mathcal{G}^*$. It is easy to see that the mappings given by $p_{_A}\longmapsto p_{_A}$ for $A\in \mathcal{G}^0$, and $s_e\longmapsto s^*_e$, $s^*_e\longmapsto s_e$ for $e\in \mathcal{G}^1$, produce an involution on the algebra $L_K(\mathcal{G})$, and for any path $\alpha = \alpha_1\cdots\alpha_n$ there exists $s^*_{\alpha} := s^*_{e_n}\cdots s^*_{e_1}$. Also, $L_K(\mathcal{G})$ has the following \textit{universal property}: if $\mathcal{A}$ is a $K$-algebra generated by a family of elements $\{b_A, c_e, c^*_e\mid A\in \mathcal{G}^0, e\in \mathcal{G}^1\}$ satisfying the relations analogous to (1) - (4)  in Definition~\ref{utraLevittpathalg}, then there always exists a $K$-algebra homomorphism $\varphi: L_K(\mathcal{G})\longrightarrow \mathcal{A}$ given by $\varphi(p_A) = b_A$, $\varphi(s_e) = c_e$ and $\varphi(s^*_e) = c^*_e$.	Furthermore, we recall a few other useful properties as follows.

\begin{lem}\label{graded}
If $\mathcal{G}$ is an ultragraph and $K$ is a field, then the Leavitt path algebra $L_K(\mathcal{G})$ has the following properties:

$(1)\ (${\cite[Theorem 2.10]{ima:tlpaou}}$)$ All elements of the set $\{p_A, s_e, s^*_e\mid A\in \mathcal{G}^0\setminus\{\varnothing\}, e\in \mathcal{G}^1\}$ are nonzero.

$(2)\ (${\cite[Theorem 2.9]{ima:tlpaou}}$)$ $L_K(\mathcal{G})$ is of the form
\begin{center}
$\textnormal{Span}_K\{s_{\alpha}p_{_A}s_{\beta}^*\mid \alpha,\beta \in \mathcal{G}^*, A\in \mathcal{G}^0 \textnormal{ and } r(\alpha)\cap A \cap r(\beta)\neq \varnothing\}$.	
\end{center}
Furthermore, $L_K(\mathcal{G})$ is a $\mathbb{Z}$-graded $K$-algebra by the grading 
\begin{center}	
$L_K(\mathcal{G})_n=\textnormal{Span}_K\{s_{\alpha}p_{_A}s_{\beta}^*\mid\alpha,\beta \in \mathcal{G}^*, A\in \mathcal{G}^0 \textnormal{ and } |\alpha| -|\beta|=n\}$\quad $(n\in \mathbb{Z})$.
\end{center}
\end{lem}

In light of Lemma~\ref{graded}, an element $x\in L_K(\mathcal{G})_n$ is called a \textit{homogeneous element of degree} $n$.

In \cite[Theorem 2.1]{abcr:tscilpaoag} Abrams, Bell, Colar and Rangaswamy provided an explicit description of a set of generators for any ideal of the Leavitt path algebra of an arbitrary graph. The following theorem is the main result of this section, which extends Abrams, Bell, Colar, and Rangaswamy's theorem to ultragraph Leavitt path algebras.

\begin{thm}\label{generatingset}
Let $K$ be a field, $\mathcal{G}$ an ultragraph and $I$ an ideal of $L_K(\mathcal{G})$.  Then there exists a generating set for $I$ consisting of elements of $I$ of the form

\[(p_A + \sum^m_{i=2}k_is^{r_i}_c)(p_A - \sum_{e\in S}s_es^*_e),\]
where $A\in \mathcal{G}^0$, $k_2, \hdots, k_m\in K$, $r_1, \hdots, r_m$ are positive integers, $S$ is a finite subset of $\mathcal{G}^1$ consisting of edges with the same source vertex $v\in A$, and, whenever $k_i\neq 0$ for some $2\le i\le m$, $c$ is the unique cycle based at $v$ such that $A\subseteq r(c)$.
\end{thm} 
\begin{proof}

Let $I$ be an ideal of $L_K(\mathcal G)$ and let $J$ be the ideal of $L_K(\mg)$ generated by all the elements of $I$ which have the form described in the statement of the theorem. We need to show that $I=J$.	We note that $I\cap \{p_A\mid A\in\mg^{0}\}\in J$ (by choosing $k_i=0$ for $2 \le i\le m$, and $S=\emptyset$).

We first claim that any element of $I$ of the form
\begin{align}\label{form1}
		(k_1s_{a_1}p_{A_1}+\cdots+k_ms_{a_m}p_{A_m})(p_A-\sum_{e \in S}s_es^*_e)
\end{align}
is in $J$, where $A_i, \hdots, A_m, A\in \mathcal{G}^0$, $k_1, \hdots, k_m\in K$, $a_i$'s are  paths in $\mg^*$, and $S$ is a finite subset of $\mathcal{G}^1$ consisting of edges with the same source vertex $v\in A$.

Towards a contradiction, suppose not. That is, assume that there is an element of $I$, of the above form, that is not in $J$. Among such elements, consider all 
	\begin{align*}
		x:= (\sum_{i=1}^mk_is_{a_i}p_{A_i})(
		p_A-\sum_{e \in S}s_es^*_e)
	\end{align*} 
for which $m$ is minimal and, among all such $x$ with minimal $m$, select one for which $(|a_1|,\hdots,|a_m|)$ is smallest in the lexicographic order of $(\bb{Z}^+)^m$ (we keep calling this element $x$). Multiplying by $k^{-1}_1$ if necessary, we may assume that
\begin{align*}
x=(s_{a_1}p_{A_1} + \sum_{i=2}^mk_is_{a_i}p_{A_i})(
p_A-\sum_{e \in S}s_es^*_e).
\end{align*} 	
Since $p_A(p_A-\sum_{e \in S}s_es^*_e) = p_A-\sum_{e \in S}s_es^*_e$, we	have that
	\begin{align*}
		x=(s_{a_1}p_{A_1\cap A} + \sum_{i=2}^mk_is_{a_i}p_{A_i\cap A})(
		p_A-\sum_{e \in S}s_es^*_e).
	\end{align*}
Since $m$ is minimal, we must have that the $s_{a_i}p_{A_i\cap A}$'s are all nonzero and pairwise distinct. We analyze the various possible cases for $x$, and show that in each case we are led to a contradiction.

{\it Case $1$}: $|a_{i}|\geq 1$ for all $1\le i\le m$.
Let $B :=\{f\in \mathcal G^1\mid s^*_fa_i\neq 0 \text{ for some } 1\le i\le m\}$. Then, $$s^*_fx=(\ds s^*_fs_{a_1}p_{A_1\cap A} +\sum_{i=2}^mk_is^*_fs_{a_i}p_{A_i\cap A})(p_A-\ds\sum_{e \in S}s_es^*_e)\in I,$$ for all $f\in B$. 
Notice that $s_f^*x\in I$ and is in the form \ref{form1}. Moreover, either $s^*_fx$ has fewer terms than does $x$, or $s^*_fx$ has the same number of terms as does $x$, in which case $(|f^*a_1|,\dots,|f^*a_m|)$ is smaller than $(|a_1|,\dots,|a_m|)$. By the minimality of $x$, we obtain that $s^*_fx\in J$ for all $f\in B$. Hence,
$s_fs^*_fx\in J$ for all $f\in B$, which yields that $x=\ds\sum_{f\in B}^{m}s_fs^*_fx\in J$, a contradiction.

{\it Case $2$}: $|a_{i}|= 0$ for some $1\le i\le m$. By the minimality assumed on $(|a_1|,\dots,|a_m|)$, this gives $ |a_1|=|a_2|=\cdots=|a_t|=0$, where $1\leq t\leq m$. Without loss of generality,  we may assume that
\begin{align*}
x=(p_{A_1} +\sum_{i=2}^tk_ip_{A_i}+\sum_{j=t+1}^{m}k_js_{a_j}p_{A_j})(p_A-\sum_{e\in S}s_es^*_e)
\end{align*}
where $A_i$'s are pairwise distinct non-empty subsets of $A$, 
$|a_j|\geq 1$ and $A_j \subseteq r(a_j)$ for all $t +1\le j \le m$. 

Consider the case when $t\geq 2$. If $t=2$, then we note that
\[p_{A_1}+k_2p_{A_2}= p_{A_1\setminus A_2}+(1+k_2)p_{A_1\cap A_2}+k_2 p_{A_2\setminus A_1}.\]
By induction on $t$, we have that
	\begin{align*}
		p_{A_1} + \sum_{i=2}^tk_ip_{A_i}=\sum_{i=1}^lk'_{i}p_{B_i},
	\end{align*}
where $k'_i\in K$, $l\in \mathbb{N}$, and $B_i$'s are non-empty elements in $\mg^{0}$ with $B_i\cap B_j =\emptyset$ for all $1\le i\neq j\le l$. Then, 
	\begin{align*}
		x=(\sum_{i=1}^lk_i'p_{B_i}+\sum^m_{j=t+1}k_js_{a_j}p_{A_j})(p_A-\sum_{e \in S}s_es^*_e).
	\end{align*}
	
For each $1\le i\le l$, we have that
\begin{align*}
		p_{B_i}x=(k'_ip_{B_i}+\sum^m_{j=t+1}k_jp_{B_i}s_{a_j}p_{A_j})(p_A-\sum_{e \in S}s_es^*_e)\in I
\end{align*}
and, for each vertex $s(a_k)$, $t+1\leq k \leq m$, such that $s(a_k)\notin \bigcup_{i=1}^l B_i$, we have that 
\begin{align*}
		p_{s(a_k)}x=(\sum^m_{j=t+1}k_jp_{s(a_k)}s_{a_j}p_{A_j})(p_A-\sum_{e \in S}s_es^*_e)\in I
\end{align*}

Notice that both $p_{B_i}x$ and $p_{s(a_k)}x$ have fewer terms than does $x$, and so we obtain that both belong to $J$. Let $F=\{s(a_k)\mid t+1\leq k \leq m \text{ and } s(a_k) \notin \bigcup_{i=1}^l B_i\}$. Then, $x=\ds\left(\sum_{i=1}^l p_{B_i} x + \sum_{v\in F} 	p_{v}x  \right) \in J$, a contradiction, and hence we must have $t=1$.

Proceeding, we consider the case when $m=1$. We note that
$x=p_{A_1}(p_A-\ds\sum_{e \in S}s_es^*_e)$, where $s(e)=v$ for each $e\in S$.
If $v\notin A_1$, then 	$x= p_{A_1}\in I$
and so $x\in J,$ a contradiction. If $v\in A_1$, then $x=p_{A_1}-\ds\sum_{e\in S}s_es^*_e\in J$, a contradiction. Therefore, we must have $m\geq 2$, which is the case we consider next. 

Assume that $m\geq 2$. Then,
$$x=(p_{A_1}+\sum_{i=2}^{m}k_is_{a_i}p_{A_i})(p_A-\sum_{e\in S}s_es^*_e),$$ where $A_i$'s are pairwise distinct non-empty subsets of $A$, $|a_i|\geq 1$ and $A_i \subseteq r(a_i)$ for all $2\le i \le m$. So,
$$p_{s(a_2)}x=(p_{s(a_2)}p_{A_1} + \ds\sum_{i=2}^mk_i  p_{s(a_2)}s_{a_i}p_{A_i})(p_A-\ds\sum_{e \in S}s_es^*_e)\in I.$$ If $s(a_2)\neq s(a_i)$, for some $3\le i\le m$, or $s(a_2)\notin A_1$, then $p_{s(a_2)}x\in J$, since $p_{s(a_2)}x$ has fewer terms than does $x$. This implies that 

\begin{align*}
x-p_{s(a_2)}x=\left(p_{A_1\setminus\{s(a_2)\}}+\sum_{i=3}^mk_i (s_{a_i}-p_{s(a_2)} s_{a_i}) p_{A_i}\right)(p_A-\sum_{e \in S}s_es^*_e)\in J,
\end{align*}
since $x-p_{s(a_2)}x$ has fewer terms than does $x$, and so $x=x-p_{s(a_2)}x+p_{s(a_2)}x\in J$, a contradiction. Hence, we must have $s(a_i)=s(a_2)\in A_1$ for all $2\le i\le m$. Let $w:= s(a_2)$. Then,
\begin{align*}
p_wx=(p_w+\sum_{i=2}^mk_is_{a_i}p_{A_i})(p_A-\sum_{e \in S}s_es^*_e)\in I.
\end{align*}
If $p_wx\in J$, then $x-p_wx=p_{A_1\setminus\{w\}}(p_A-\sum_{e \in S}s_es^*_e)\in J$, and so $x = x -p_wx + p_wx\in J$, a contradiction, showing that $p_wx \notin J.$ 

%Furthermore, we have \begin{align*}p_wxp_w=\left(p_w+\sum_{i=2}^mk_is_{a_i}p_{A_i}\right)\left(p_A-\sum_{e \in S}s_es^*_e\right)p_w.\end{align*}

Consider the case when $v\neq w$. Then, since $p_w(p_A-\sum_{e \in S}s_es^*_e) = p_w = (p_A-\sum_{e \in S}s_es^*_e)p_w$, we have that
\begin{align*}
		p_wxp_w=(p_w+\sum_{i=2}^mk_is_{a_i}p_{A_i} p_w)(p_A-\sum_{e \in S}s_es^*_e) \in I.
\end{align*}
If $w\notin A_i$ for some $2\le i\le m$, then we have that $p_wxp_w\in J$ (since $p_wxp_w$ has fewer terms than does $x$) and \begin{align*}
		p_wx-p_wxp_w=(\sum_{i=2}^mk_is_{a_i}p_{A_i\setminus\{w\}})(p_A-\sum_{e\in S}s_es^*_e)\in J
	\end{align*}
(since it has fewer terms than does $x$), and so $p_wx\in J$, a contradiction. This shows that $w\in A_i$ for all $2\le i\le m$. Furthermore, if $p_w x p_w \in J$ then, by the equality above, we obtain that $p_w x \in J$, since $p_w x= (p_w x - p_w x p_w)+p_w x p_w$. This is a contradiction, and hence we conclude that 
\begin{align*}
p_{w}xp_w= p_w+\sum_{i=2}^mk_is_{a_i}p_w\in I\setminus J.
\end{align*} 

There are now three subcases to consider; we obtain a contradiction in each.
	
Firstly, suppose that there are no cycles based at $w$. Then, $p_wxp_w = p_w\in J$, a contradiction.

Secondly, suppose that there is an unique cycle $c$ based at $w$. For each $2\le i\le m$, we conclude that $a_i=c^{r_i}$ for some positive integer $r_i$, and so
		\begin{align*}
			p_wxp_w=(p_w+\sum_{i=2}^mk_i(s_c)^{r_i})(p_w + 0)
		\end{align*}
has the indicated form of the generators of $J$. This implies that $p_wxp_w\in J$, a contradiction.
		
Thirdly, suppose that there are at least two distinct cycles  based at $w$, say $c$ and $d$, and we have $s^*_cs_d=0 = s^*_d s_c$. Then for some positive integer $n$, where $|c^n|>|a_i|$ for all $2 \le i\le m$, we get
\[(s_c^n)^* (p_wxp_w) s_c^n =p_{r(c)}+\sum_{i=2}^mk_i(s_c^n)^*s_{a_i}s_c^n\in I.\]
If $(s_c^n)^*a_is_c^n\neq 0$, then $(s_c^n)^*a_i\neq0$ and, as $|c^n|>|a_i|$, we obtain that $c^n=a_ib_i$ for some path $b_i\in \mathcal{G}^*$. Whence, since $(s_c^n)^*a_i\neq 0$, we obtain that $a_i=c^{n_i}$ for some positive integer $n_i < n$. Since $s^*_d s_c = 0$, for every $i$ one gets $s^*_d(s_c^n)^*s_{a_i}s_c^ns_d = 0$, and so
\[p_{r(d)} =s^*_d p_{r(c)} s_d= s_d^*(s_c^n)^* (p_w xp_w)s_c^ns_d\in I.\]
Since $I\cap \{p_A\mid A\in \mg^{0}\}\subseteq J$, we have $p_{r(d)}\in J$, and so $p_w = p_w p_{r(d)}\in J$ and $p_wxp_w= p_w(p_wxp_w)\in J$, a contradiction.
		
In any of the three cases, we arrive at a contradiction, and so we must have $v=w$, which is the case that we deal with next.

We proceed similarly to what we did above for $w$. Notice that \[p_v x = (p_v+\sum_{i=2}^mk_is_{a_i}p_{A_i} )(p_A-\sum_{e \in S}s_es^*_e)\in I\setminus J.\] Since $p_v(p_A-\sum_{e \in S}s_es^*_e) = (p_A-\sum_{e \in S}s_es^*_e)p_v = p_v-\sum_{e \in S}s_es^*_e$, we obtain that \[p_v xp_v = (p_v+\sum_{i=2}^mk_is_{a_i}p_{A_i}p_v )(p_A-\sum_{e \in S}s_es^*_e) \in I.\]
If $v\notin A_i$ for some $2\le i\le m$, then 
$p_vxp_v \in J$ (since $p_vxp_v$ has fewer terms than does $x$), and so $$p_vx-p_vxp_v = (\sum_{i=2}^mk_is_{a_i}p_{A_i\setminus\{v\}})(p_A-\sum_{e\in S}s_es^*_e)\in J,$$ and hence $p_vx\in J$, a contradiction. This implies that $v\in A_i$ for all $2\le i\le m$, and $$p_vxp_v =(p_v+\sum_{i=2}^mk_is_{a_i})(p_v-\sum_{e \in S}s_es^*_e)\in I\setminus J.$$
There are three subcases to consider; we obtain a contradiction in each.

Firstly, suppose that there are no cycles based at $v$. Then, $p_vxp_v = (p_v +0)(p_v-\sum_{e \in S}s_es^*_e)\in J$, a contradiction.

Secondly, suppose that there is an unique cycle $c$ based at $w$. For each $2\le i\le m$, we conclude that $a_i=c^{r_i}$ for some positive integer $r_i$, and so
\begin{align*}
p_vxp_v=(p_v + \sum_{i=2}^mk_i(s_c)^{r_i})(p_v-\sum_{e \in S}s_es^*_e)
\end{align*}
has the indicated form for the generators of $J$. This implies that $p_vxp_v\in J$, a contradiction.

Thirdly, suppose that there are at least two distinct cycles  based at $v$. Let $F:= \{f\in \mathcal{G}^1\mid s^*_f s_{a_i}\neq 0 \text{ for some } 2\le i\le m\}$. Assume that $F \cap S \neq \emptyset$. Let $f\in F \cap S$. We have $s_fs^*_f (p_v - \sum_{e \in S}s_es^*_e) = 0$, and  $$s_fs^*_f p_vxp_v = (s_fs^*_f + \sum_{i=2}^mk_is_fs^*_fs_{a_i})(p_v - \sum_{e \in S}s_es^*_e) = (\sum_{i=2}^mk_is_fs^*_fs_{a_i})(p_v - \sum_{e \in S}s_es^*_e).$$ Notice that $s_fs^*_fs_{a_i}$ is either $0$ or $s_{a_i}$, and so the displayed expression for $s_fs^*_f p_vxp_v$ has the correct form. Hence, $s_fs^*_f p_vxp_v\in J$ by the minimality of $m$. Then,
\[p_vxp_v - s_fs^*_f p_vxp_v = (p_v + \sum_{\{a_i\mid s^*_fa_i=0\}}k_is_{a_i})(p_v - \sum_{e \in S}s_es^*_e)\]
is also of the correct form, and the left-hand factor has fewer than $m$ nonzero terms, so $p_vxp_v - s_fs^*_f p_vxp_v\in J$ by the minimality of $m$. This implies that $p_vxp_v = (p_vxp_v - s_fs^*_f p_vxp_v) + s_fs^*_f p_vxp_v\in J$, a contradiction. 

Thus we must have $F \cap S = \emptyset$. We then have $s^*_e s_{a_2} =0$ for all $e\in S$. Let $c$ be a cycle in $\mg$ based at $v$  having the same initial edge as does $a_2$ (such a cycle exists because $a_i$, $i=2,\ldots m$, must be a closed path, otherwise $p_v s_{a_i} p_v=0$). We get that $s^*_es_c = 0$ for all $e\in S$, and $(p_v - \sum_{e \in S}s_es^*_e)s_c = s_c$.

By the hypothesis of this subcase, there exists a cycle $d$ based at $v$ such that $s^*_ds_c=0 = s^*_c s_d$. Then, for some positive integer $n$, with $|c^n|>|a_i|$ for all $2 \le i\le m$, we obtain that
\[(s_c^n)^* (p_vxp_v) s_c^n =p_{r(c)}+\sum_{i=2}^mk_i(s_c^n)^*s_{a_i}s_c^n\in I.\]
If $(s_c^n)^*a_is_c^n\neq 0$, then $(s_c^n)^*a_i\neq 0$, and as $|c^n|>|a_i|$, $c^n=a_ib_i$ for some path $b_i\in \mathcal{G}^*$. Whence, $a_i=c^{m_i}$ for some positive integer $m_i < n$. Since $s^*_d s_c = 0$, for every $i$ one gets  $s^*_d(s_c^n)^*s_{a_i}s_c^ns_d = 0$, and so
\[p_{r(d)} =s^*_d p_{r(c)} s_d= s_d^*(s_c^n)^* (p_w xp_w)s_c^ns_d\in I.\]
Since $I\cap \{p_A\mid A\in \mg^{0}\}\subseteq J$, we have that $p_{r(d)}\in J$, and so $p_v = p_v p_{r(d)}\in J$ and $p_vxp_v= p_v(p_vxp_v)\in J$, the final contradiction required to show the claim.

\medskip\medskip

Now we prove that any arbitrary element of $I$ is in $J$. Again, aiming at a contradiction, suppose that $I\setminus J\neq \emptyset$ and let 
	\begin{align*}
		x=(\sum_{i=1}^mk_is_{a_i}p_{A_i}s^*_{b_i})(p_A-\sum_{e \in S}s_es^*_e) \in I\setminus J,
	\end{align*}
where $k_i\in K$, $a_i, b_i \in \mathcal{G}^*$, $A_1, \hdots, A_m, A\in \mg^{0}$, $S$ is a finite subset of $\mathcal{G}^1$ consisting of edges with the same source vertex $v\in A$, and $m$ is minimal. As before, we may choose $k_1 = 1$. Among all such $x$, select one for which
$(|b_1|,|b_2|,\dots,|b_m|)$ is smallest in the lexicographic order of $(\bb{Z}^+)^m$.
%We assume that $s(a_i)=s(a_j)$ and $s(b_i)=s(b_j)=v$ for every $i,j$.

We note that $p_A(p_A - \sum_{e \in S}s_es^*_e) = p_A - \sum_{e \in S}s_es^*_e$, and so $$x=(\sum_{i=1}^mk_is_{a_i}p_{A_i}s^*_{b_i}p_A)(p_A-\sum_{e \in S}s_es^*_e).$$ Then, by the minimality of $m$, we have that $s(b_i)\in A$ if $|b_i| > 0$, and $s(b_i) \cap A\neq \emptyset$ if $|b_i|=0$ (in fact, we may assume, without loss of generality, that $s(b_i) \subseteq A$ if $|b_i|=0$). 

Suppose that $|b_i|>0$ for some $1\le i\le m$, and write $b_i=e_ib'_i$ for some $e_i\in \mg^1$ and $b'_i\in \mathcal{G}^*$. If $e_i\in S$, then 
\[s^*_{b_i}(p_A-\sum_{e \in S}s_es^*_e)= s^*_{b'_i}s^*_{e_i}(p_A-\sum_{e \in S}s_es^*_e) =s^*_{b'_i}s^*_{e_i}-s^*_{b'_i}s_{e_i}^*=0.\] So, we may assume that if $|b_i| > 0$  in the indicated expression for $x$, then $e_i\notin S$.
	
Now, assume that $|b_i|>0$ for all $1\le i\le m$. 
As above, write $b_i=e_ib'_i$. Note that, for any edge $f\in s^{-1}(v)\setminus S$, we have that $(p_A-\sum_{e \in S}s_es^*_e)s_f=s_f$. Therefore, for any $f\in s^{-1}(A)\setminus S$, we have that
\begin{align*}
xs_f&=(\sum_{i=1}^mk_is_{a_i}p_{A_i}s_{b_i}^*)(p_A-\sum_{e \in S}s_es^*_e)s_f=(\sum_{i=1}^mk_is_{a_i}p_{A_i}s_{b_i}^*)s_f=\\
&=\sum_{\{i\mid e_i=f\}}k_is_{a_i}p_{A_i}s_{b'_i}^*\in I.
\end{align*}
If the number of monomial terms in $xs_f$ is less than $m$, then $xs_f\in J$. If the number of monomial terms in $xs_f$ is  $m$ then, since $(|b_1'|,|b'_2|,\dots,|b'_n|)<(|b_1|,|b_2|,\dots,|b_n|)$, the minimal condition implies that $xs_f\in J$. In particular, for each $e_i$ which appears as the initial edge of some $b_i$ in the expression for $x$, we have that $xs_{e_i}s^*_{e_i}\in J$. But this in turn yields
\begin{align*}
x=\sum_{\{\text{distinct } e_j\mid 1\leq j\leq m\}}xs_{e_j}s^*_{e_j} \in J,
\end{align*}
%{\color{red} for the above equality I believe we need the $f$ to vary in $s^{-1}(A)$, as I propose above.}
a contradiction.
		
On the other hand, suppose that $|b_i|=0$ for some $1\leq i \leq m$. Without loss of generality, assume that $|b_1|=\cdots=|b_u|=0$ for some $1\leq u\le m$ and that $|b_i|>0$ for some $m\geq i\geq u+1$ (notice that we have already dealt with the case in which $|b_i|=0$ for all $i$ in the first part of this proof). Then,
\begin{align*}
x=(\sum_{i=1}^uk_is_{a_i}p_{A_i}+\sum_{j=u+1}^mk_js_{a_j}p_{A_j}s_{b_j}^*)(p_A-\sum_{e \in S}s_es^*_e).
\end{align*}
Let $T:=\{f\in\mg^1\mid s^*_{b_i}s_f\neq 0 \text{ for some } u+1\leq i\leq m\}$, that means, $T$ is the set of edges which appear as the initial edge of some path $b_i, u+1\leq i\leq n$. As indicated above, by minimality we may assume that $T\cap S = \emptyset$. Again using minimality, an argument analogous to one used previously yields that $xs_f \in J$ for all $f\in T$, and so $$\ds\sum_{f\in T}xs_fs^*_f\in J.$$ Write $b_i=e_{i}b'_i$ for each $u+1\leq i\leq m$. Then, for $f\in T$, we have that $s_{b_j}^*s_fs^*_f = 0$, unless $e_i=f$ in which case $s_{b_j}^*s_fs^*_f = s_{b_j}^*$. This yields that 
$s_{b_j}^*(p_A-\sum_{f\in T} s_fs^*_f)=s^*_{b_j}-s^*_{b_j}=0$ for all $ u+1\leq i\leq m$, and so
		\begin{align*}
			x-\sum_{f \in T}xs_fs^*_f=(\sum_{i=1}^uk_is_{a_i}p_{A_i})(p_A-\sum_{g\in S\cup T} s_gs^*_g)\in I.
		\end{align*}
		
%{\color{blue} Why $g\in T$ is such that $s(g)=v$? To use the claim below we need all edges $g$ with the same vertex source.... Check this. Below in red I propose a way around this. If ok, have to erase the last paragraph in black, after the red part }

Let $U=\{s(e): e \in S\cup T\}$. Since $U$ is finite, enumerate it, say $U=\{u_1, \ldots u_k\}$. For all $u\in U$, let $X_u=\{e\in S\cup T: s(e)=u\}$. 
Then, for each $u\in\{u_1,\ldots u_{k-1}\}$, we have that \begin{align*}
			(x-\sum_{f \in T}xs_fs^*_f)p_u=(\sum_{i=1}^uk_is_{a_i}p_{A_i})(p_u-\sum_{g\in X_u} s_gs^*_g)\in I,
		\end{align*}
and, by the claim proved in the first part of this proof, we have that $(x-\sum_{f \in T}xs_fs^*_f)p_u \in J$. Analogously, we obtain that \begin{align*}
			(x-\sum_{f \in T}xs_fs^*_f)p_{A\setminus \{u_1,\ldots,u_{k-1}\}}=(\sum_{i=1}^uk_is_{a_i}p_{A_i})(p_{A\setminus \{u_1,\ldots,u_{k-1}\}}-\sum_{g\in X_{u_k}} s_gs^*_g)\in I,
		\end{align*}
and, again by the claim proved in the first part of this proof, we have that $(x-\sum_{f \in T}xs_fs^*_f)p_{A\setminus \{u_1,\ldots,u_{k-1}\}} \in J$.
We conclude that
	\begin{align*}
			x-\sum_{f \in T}xs_fs^*_f=
			(x-\sum_{f \in T}xs_fs^*_f)p_{A\setminus \{u_1,\ldots,u_{k-1}\}}
			+ \sum_{u\in\{u_1,\ldots,u_{k-1}\}}(x-\sum_{f \in T}xs_fs^*_f)p_u
			\in J,
		\end{align*}
which gives $x\in J$, a contradiction, thus finishing the proof.
\end{proof}

\subsection{Applications}\label{sec:appl}
%The results regarding purely infinite simple ultragraph Leavitt path algebras could be derived by general groupoid results (\cite[Theorem 3.2]{bch:dsopisga} and \cite[Theorem 5.5]{dgv:uavlggwatgut}). The main goal of this section is both to give a graph-theoretic criterion for purely infinite simple ultragraph Leavitt path algebras (Theorem~\ref{purelysimple}) and provide a complete description of graded simple ultragraph Leavitt path algebras (Theorem~\ref{TrichotomyPrinciple}). Recall (see e.g. \cite{agp:kopisrr}) that an idempotent $e$ in a ring $R$ is called \textit{infinite} if $eR$ is isomorphic as a right $R$-module to a proper direct summand of itself. $R$ is called \textit{purely infinite} in case every nonzero right ideal of $R$ contains an infinite idempotent. The following lemma provides us with a useful criterion for purely infinite simple rings with local units.

The structure of graded ideals of ultragraph Leavitt path algebras was described in \cite[Theorem 3.4]{ima:tlpaou}. We may invoke Theorem~\ref{generatingset} to obtain information about generating sets for graded ideals of ultragraph Leavitt path algebras. In particular, Theorem~\ref{generatingset} allows us to provide a  more direct proof of the key piece of \cite[Theorem 3.4]{ima:tlpaou}. Before doing so, we need to recall  some useful notions (see, e.g. \cite[Definition 3.1]{tomf:soua} and \cite[Definition 2.2]{ima:tlpaou}). 

Let $\mathcal{G}$ be an ultragraph. A subset $\mathcal{H}\subseteq\mathcal{G}^0$ is called \textit{hereditary} if the following conditions are satisfied:
\begin{itemize} 
	\item[(1)] if $e$ is an edge with $\{s(e)\} \in\mathcal{H}$, then $r(e)\in \mathcal{H}$;
	
	\item[(2)] $A\cup B \in \mathcal{H}$ for all $A, B \in \mathcal{H}$;
	
	\item[(3)] $A\in \mathcal{H}, B\in \mathcal{G}^0$ and $B\subseteq A$, imply that $B\in\mathcal{H}$.
\end{itemize}	
A subset $\mathcal{H}\subseteq\mathcal{G}^0$ is called \textit{saturated} if for any $v\in G^0$ with $0<|s^{-1}(v)|<\infty$, we have that
\begin{center}
	$\left\{r(e)\mid e\in \mathcal{G}^1 \text{ and } s(e)=v\right\}\subseteq \mathcal{H}$ implies $\left\{v\right\}\in\mathcal{H}$.	
\end{center} 

For a saturated hereditary subset $\mathcal{H}$ of $\mathcal{G}^0$, we define the {\it breaking vertices} of $\mathcal{H}$ to be the set
\[B_{\mathcal{H}}:=\{v\in G^0\mid |s^{-1}(v)|=\infty \text{ and } 0 < |s^{-1}(v) \cap \{e\mid r(e)\notin \mathcal{H}\}|<\infty\},\]
and for any $v\in B_{\mathcal{H}}$ we let \[p^{\mathcal{H}}_v := p_v-\ds\sum_{e\in s^{-1}(v),\ r(e)\notin \mathcal{H}}s_es^*_e.\]

An {\it admissible pair} of $\mg$ is a pair $(\mathcal{H}, S)$ consisting of a saturated hereditary subset $\mathcal{H}$ of $\mg^0$ and a subset $S\subseteq B_{\mathcal{H}}$.

Let $K$ be a field and $I$ be an ideal of $L_K(\mg)$. We let  $\mathcal{H}_I := \{A\in \mg^0\mid p_A \in I\}$, which is a saturated hereditary subset of $\mg^0$, and we let $S_I :=\{v\in B_{\mathcal{H}_I}\mid p^{\mathcal{H}_I}_v\in I\}$.

%{\color{red} Should we give somewhere the definition of a graded ideal?}

\begin{thm}\label{gi-generatingset}
Let $K$ be a field, $\mg$ an ultragraph, and $I$ an ideal of $L_K(\mg)$. Then, the following are equivalent:

$(1)$ $I$ is a graded ideal;

$(2)$ $I$ is generated by elements of the form $p_A -\ds\sum_{e\in S}s_es^*_e$, where $A\in \mg^{0}$ and $S$ is a finite subset of $\mathcal{G}^1$ consisting of edges with the same source vertex $v\in A$;

$(3)$ $I$ is generated by the subset  $$\{p_A\mid A\in \mathcal{H}_I\} \cup \{p^{\mathcal{H_I}}_v \mid v\in S_I\}.$$ 
\end{thm}
\begin{proof}
(1)$\Longrightarrow$(2). By Theorem \ref{generatingset}, $I$ is generated as an ideal by elements in $I$ of the form	
\[x= (p_A-\sum_{i=2}^mk_is_c^{r_i})(p_A-\sum_{e \in S}s_es^*_e)=(p_A-\sum_{e \in S}s_es^*_e) -\sum_{i=2}^{m}k_is_c^{r_i}(p_A-\sum_{e \in S}s_es^*_e),\]	
where $A\in \mathcal{G}^0$, $k_2, \hdots, k_m\in K$, $r_1, \hdots, r_m$ are positive integers, $S$ is a finite subset of $\mathcal{G}^1$ consisting of edges with the same source vertex $v\in A$, and, whenever $k_i\neq 0$ for some $2\le i\le m$, $c$ is the unique cycle based at $v$ such that $A\subseteq r(c)$. Since $I$ is graded, each of the graded components of $x$ is in $I$. Since $\deg(p_A-\sum_{e \in S}s_es^*_e) = 0$, we have that the degree $0$ component of $x$ is $p_A-\sum_{e \in S}s_es^*_e$, while the degree $|c|r_i$ component of $x$ for $r_i\ge 1$ is $k_i s^{r_i}_c(p_A-\sum_{e \in S}s_es^*_e)$. This shows that
$x$ belongs to the ideal generated by elements in $I$ of the form $p_A-\sum_{e \in S}s_es^*_e$, as desired.

%{\color{red} I rewrote a bit the proof (2) $\rightarrow$ (3) below. Please check}

(2)$\Longrightarrow$(3). Consider a generator $y=p_A-\sum_{e\in S}s_es^*_e\in I$, where $A\in \mg^{0}$ and $S$ is a finite subset of $\mathcal{G}^1$ consisting of edges with the same source vertex $v\in A$. 

If $S=\emptyset$, then $y=p_A\in I$, and so $A\in \mathcal{H}_I$.

Let $J$ be the ideal of $L_K(\mg)$ generated by $\{p_A\mid A\in \mathcal{H}_I\}$. Assume that $y$ is not in $J$. We show next that, in this case, $v$ is a breaking vertex for $\mathcal{H}_I$.

Notice that, 
%We then have $y = p_{A\setminus\{v\}} + (p_v-\sum_{e\in S}s_es^*_e)$ and
$$p_{A\setminus\{v\}} = p_{A\setminus\{v\}} (p_A-\sum_{e\in S}s_es^*_e) =  p_{A\setminus\{v\}} y\in I,$$ which shows that $A\setminus\{v\}\in \mathcal{H}_I$, and hence $p_{A\setminus\{v\}} \in J$.
Furthermore, we note that $A\notin \mathcal{H}_I$, since otherwise $p_A\in J$, and so $y = p_A(p_A-\sum_{e\in S}s_es^*_e) \in J$, which contradicts our hypothesis.

From the above we obtain that $\{v\}\notin \mathcal{H}_I$, since otherwise, if $\{v\}\in \mathcal{H}_I$, then $p_v\in I$ and hence $p_A = p_v + p_{A\setminus\{v\}} \in I$, that is, $A\in \mathcal{H}_I$, a contradiction.%Thus, we must have $\{v\}\notin \mathcal{H}_I$. 

Let $S'=\{e\in S\mid r(e)\in \mathcal{H}_I\}$. Notice that, for any $e$ such that $r(e)\in \mathcal{H}_I$, we have that  $s_e=s_ep_{r(e)}\in J$, and so $s_es^*_e\in J$. 
Hence, if $S'=S$ then $\sum_{e \in S}s_es^*_e\in J$ and $p_A=y +\sum_{e \in S}s_es^*_e\in I$, and so $A\in \mathcal{H}_I$, a contradiction. So, $S'\subsetneq S$. Moreover, if there exists an edge $f\in \mg^1$ with $s(f)=v$ and $r(f)\notin \mathcal{H}_I$, then $f$ must belong to $S$, because otherwise 	
\[z:=(p_A - s_fs^*_f)(p_A-\sum_{e \in S}s_es^*_e) =p_A-\sum_{e \in S}s_es^*_e-s_fs^*_f\in I,\]
which shows that $s_fs^*_f = y -z\in I$. Then, $p_{r(f)}=s_f^*(s_fs^*_f)s_f\in I$, and so $r(f)\in \mathcal{H}_I$, a contradiction. Thus, we obtain that
$0<|s^{-1}(v)\cap\{e\mid r(e)\notin \mathcal{H}_I\}|<\infty$.

Suppose that $s^{-1}(v)$ is finite. Let $y'=y + \sum_{e\in S'} s_es_e^*= p_A - \sum_{e\in S: r(e)\notin \mathcal{H}_I} s_es_e^* $. Then $y'\notin J$, since otherwise $y = y'-\sum_{e\in S'} s_es_e^*$ belongs to $J$. Then,  
\[y'= p_{A\setminus\{v\}}+p_v-\sum_{e\in S: r(e)\notin \mathcal{H}_I} s_es_e^*=p_{A\setminus\{v\}}+\sum_{e\in s^{-1}(v),\ r(e)\in \mathcal{H}_I}s_es^*_e\in J,\]  a contradiction. We conclude that $v$ is an infinite emitter and hence it is a breaking vertex for $\mathcal{H}_I$, as desired.

Finally, notice that we can write
\[y=p_{A\setminus\{v\}}+p_v^{\mathcal{H}_I}- \sum_{e\in s^{-1}(v),\ r(e)\in \mathcal{H}_I} s_es^*_e,\] and, since both $p_{A\setminus\{v\}}$ and $\sum_{e\in s^{-1}(v),\ r(e)\in \mathcal{H}_I} s_es^*_e$ belong to $J$, $(3)$ is proved.

% If $r(e)\in \mathcal{H}_I$ for any $e\in S$, we have $s_e=s_ep_{r(e)}\in J$, and so $s_es^*_e\in J$. Subtracting from $y$ all those terms $s_es^*_e$ for which $r(e)\in \mathcal{H}_I$, we may assume that $r(e)\notin \mathcal{H}_I$ for all $e\in S$.	We may assume that this process will not exhaust all of $S$, since otherwise $\sum_{e \in S}s_es^*_e\in J$ and $p_A=y -\sum_{e \in S}s_es^*_e\in I$, and so $A\in \mathcal{H}_I$, a contradiction. If there exists an edge $f\in \mg^1$ with $s(f)=v$ and $r(f)\notin \mathcal{H}_I$, then $f$ must belong to $S$, because otherwise 	
% \[z:=(p_A - s_fs^*_f)(p_A-\sum_{e \in S}s_es^*_e) =p_A-\sum_{e \in S}s_es^*_e-s_fs^*_f\in I,\]
% showing that $s_fs^*_f = y -z\in I.$ We then have $p_{r(f)}=s_f^*(s_fs^*_f)s_f\in I$, and so $r(f)\in \mathcal{H}_I$, a contradiction. Thus we obtain that $$S = \{e\in s^{-1}(v)\mid r(e)\notin \mathcal{H}_I\}.$$
% If $s^{-1}(v)$ is finite, then 
% \[y= p_{A\setminus\{v\}}+p_v-\sum_{e\in S}s_es^*_e=p_{A\setminus\{v\}}+\sum_{e\in s^{-1}(v),\ r(e)\in \mathcal{H}_I}s_es^*_e\in J,\]  a contradiction. Therefore, $v$ is an infinite emitter, and so $v\in B_{\mathcal{H}_I}$, proving $(3)$.

(3)$\Longrightarrow$(1). This is straightforward since any ideal generated by homogeneous elements is graded.
\end{proof}

Consequently, we have the following.

\begin{cor}\label{gi-cor}
Let $\mg$ be an ultragraph and $K$ a field. Then, the following statements are true:

$(1)$ $I^2=I$ for all  graded ideal $I$ of $L_{K}(\mg)$;

$(2)$ For every nonzero graded ideal $I$ of $L_{K}(\mg)$, there exists a vertex $v \in G^{0}$ such that $p_{v}\in I$;

$(3)$ If every cycle in $\mg$ has an exit, then for every nonzero ideal $I$ of $L_{K}(\mg)$, there exists a vertex $v \in G^{0}$ such that $p_{v}\in I$.
\end{cor}
\begin{proof} Let $I$ be a graded ideal of $L_{K}(\mg)$. By Theorem \ref{gi-generatingset}, $I$ is generated by the subset  $\{p_A\mid A\in \mathcal{H}\} \cup \{p_v-\sum_{e\in s^{-1}(v),\ r(e)\notin \mathcal{H}}s_es^*_e\mid v\in B_{\mathcal{H}}\},$ where $\mathcal{H} := \{A\in \mg^0\mid p_A \in I\}$. Then, (1) follows from the fact that both $p_A$ and $p_v-\sum_{e\in s^{-1}(v),\ r(e)\notin \mathcal{H}}s_es^*_e$ are idempotent in $L_K(\mg)$.	

If $I$ is a nonzero graded ideal of $L_{K}(\mg)$, then $I$ contains either $p_A$ for some $A\in \mathcal{G}^0\setminus \{\emptyset\}$, or $x:= p_v-\sum_{e\in s^{-1}(v),\ r(e)\notin \mathcal{H}}s_es^*_e$ for some $v\in B_{\mathcal{H}}$. In the first case, we have $p_v = p_v p_A \in I$ for all $v\in A$, and we are done. In the latter case,  since $v$ is an infinite emitter, there exists an edge $f\in s^{-1}(v)$ such that $r(f)\in \mathcal{H}$. Then, \[s^*_fxs_f=s^*_fp_vs_f-\sum_{e \in s^{-1}(v), \ r(e)\notin \mathcal{H}} s^*_fs_es^*_es_f=p_{r(f)}\in I.\]
So, $p_w \in I$ for all $w\in r(f)$, showing (2). We prove (3) below.

Assume that every cycle in $\mg$ has an exit and $I$ is a nonzero ideal of $L_K(\mg)$. By Theorem \ref{gi-cor}, $I$ contains a nonzero element of the form 
		$$x:=(p_A+\ds\sum k_is_c^{r_i})(p_A-\sum_{e \in S}s_es^*_e),$$
where $A\in \mathcal{G}^0$, $k_2, \hdots, k_m\in K$, $r_1, \hdots, r_m$ are positive integers, $S$ is a finite subset of $\mathcal{G}^1$ consisting of edges with the same source vertex $v\in A$, and, whenever $k_i\neq 0$ for some $2\le i\le m$, $c$ is the unique cycle based at $v$ such that $A\subseteq r(c)$. Consider the following two cases:
	
{\it Case 1:} $\{v\} \varsubsetneq A$. We then have that $A\setminus \{v\}\in \mathcal{G}^0\setminus \{\emptyset\}$ and
	$$p_{A\setminus \{v\}}=p_{A\setminus \{v\}}(p_A+\sum k_is_c^{r_i})(p_A-\ds\sum_{e \in S}s_es^*_e) =  p_{A\setminus \{v\}}x\in I.$$ So, $p_w=p_w p_{A\setminus \{v\}}\in I$ for all $w\in A\setminus \{v\}$, as desired.
		
{\it Case 2:} $\{v\} = A$. If $v$ is a regular vertex, then we have that
$$x= (p_v-\sum k_is_c^{r_i})(p_v-\sum_{e \in S}s_es^*_e)=(p_v-\sum k_is_c^{r_i})(\sum_{e \in s^{-1}(v)\setminus S}s_es^*_e)\neq 0.$$
So, $s^{-1}(v)\setminus S\neq \emptyset$ and hence there exists an edge $f\in s^{-1}(v)\setminus S$. If $v$ is an infinite emitter, then it is obvious that  there exists an edge $f\in s^{-1}(v)\setminus S$. Therefore, in any case, there exists an edge $f\in s^{-1}(v)\setminus S$. Then,
\begin{align*}
s^*_fxs_f&=s^*_f(p_v+\sum k_is_c^{r_i})(p_v -\sum_{e \in S}s_es^*_e)s_f =s^*_f(p_v+\sum k_is_c^{r_i})s_f\\&=p_{r(f)}+\sum k_is_f^*s_c^{r_i}s_f\in I.\end{align*}
Write $c = e_1 \cdots e_n$.	If $f\neq e_1$, then we have $s^*_fs^{r_i}_cs_f=0$, and so $p_{r(f)}\in I$. Hence $p_u= p_u r(f) \in I$ for all $u\in r(f)$, as desired.
		
Consider the case when $f=e_1$. We then have that $$s^*_fxs_f =p_{r(e_1)}+\ds\sum k_is^{r_i}_g \in I,$$ where $g=g_1g_2\cdots g_n$ with $g_1=e_2,\ g_2 = e_3,\dots, g_{n-1}=e_n, g_n=e_1$. Let $y:= p_{r(e_1)}+\ds\sum k_is_g^{r_i}\in I$. By our hypothesis, $g$ has an exit, that is, there is either an edge $e\in \mathcal{G}^1$ such that there exists an $1\le i\le n$ for which $s(e)\in r(g_i)$ but $e\ne g_{i+1}$, or a sink $w$ such that $w\in r(g_i)$ for some $1\le i\le n$. 

In the first case, for $h:= g_1 \cdots g_ie$, we have that
$$s^*_{h}ys_{h}=s^*_{h}p_{r(e_1)}s_h=p_{r(h)}\in I.$$
So, $p_w = p_w p_{r(e_1)} \in I$ for all $w\in r(h)$, as desired. In the second case, for $h:= g_1 \cdots g_iw$, we have that $s^*_{h}ys_{h}=p_w\in I$, as desired.

In any case, we obtain that $p_v \in I$ for some $v\in G^0$, thus the proof is finished.
\end{proof}

Before we proceed, we provide below a class of ultragraph Leavitt path algebras that can not be realized as the Leavitt path algebra of a graph. 

\begin{prop}\label{ExULPaNLPA}
Let $\mathcal G$ be an ultragraph such that $L_K(\mathcal G)$ is unital and such that there is an infinite number of hereditary saturated subsets in $\mathcal G^0$. Then, $L_K(\mathcal G)$ is not isomorphic to any graph Leavitt path algebra.
\end{prop}
\begin{proof}
Suppose that $L_K(\mathcal G)$ is isomorphic to $L_K(E)$, where $E$ is a graph. Since $L_K(\mathcal G)$ is unital, $L_K(E)$ must also be unital, and hence the graph $E$ has finitely many vertices. Therefore, by \cite[Theorem3.4]{ima:tlpaou}, $L_K(E)$ has only a finite number of graded ideals. On the other hand, again by \cite[Theorem3.4]{ima:tlpaou}, $L_K(\mathcal G)$ has an infinite number of graded ideals, a contradiction.  
\end{proof}

\begin{rem} The example given in \cite[Example~5.11]{ima:tlpaou} of an ultragraph Leavitt path algebra that can not be realized as a Leavitt path algebra of a graph was over the base field $\mathbb Z_2$. Our result above shows that this conclusion is true regardless of the base field (and of course allows the construction of many other examples).
\end{rem}

The graded-uniqueness theorem of ultragraph Leavitt path algebras was established in \cite[Theorem 2.14]{ima:tlpaou} and \cite[Theorem 5.4]{dgv:uavlggwatgut} in terms of two different approaches. We may recover this theorem using Corollary \ref{gi-cor}.

\begin{thm}[{cf. \cite[Theorem 2.14]{ima:tlpaou} and \cite[Theorem 5.4]{dgv:uavlggwatgut}}]\label{gutheo}
Let $\mg$ be an ultragraph, $K$ a field, $A$ a $\bb{Z}$-graded $K$-algebra and $\vp:L_K(\mg)\to A$  a $\bb{Z}$-graded homomorphism. Then, $\vp$ is injective if and only if $\vp(p_v)\neq 0$ for all $v\in G^0$.
\end{thm}
\begin{proof}
($\Longrightarrow$) Assume that $\vp$ is injective. By Lemma \ref{graded} (1), we have $p_v\neq 0$ for all $v\in G^0$,  and so $\vp(p_v)\neq 0$ for all $v\in G^0$.

($\Longleftarrow$) Assume that $\vp(p_v)\neq 0$ for all $v\in G^0$, and $\ker(\vp)\neq \{0\}$. Since $\vp$ is a $\bb{Z}$-graded homomorphism, we have that $\ker(\vp)$ is a $\bb{Z}$-graded ideal of $L_K(\mg)$. By Corollary~ \ref{gi-cor} (2), $\ker(\vp)$ contains an element $p_v$, for some $v\in G^0$. This implies that $\vp(p_v)=0$, a contradiction. Thus, $\ker(\vp)= \{0\}$ and $\vp$ is injective, as desired.	
\end{proof}

The Cuntz-Krieger uniqueness theorem of ultragraph Leavitt path algebras was established in \cite[Theorem 2.17]{ima:tlpaou}. We may recover this theorem using Theorem~\ref{generatingset}.

\begin{thm}[{cf. \cite[Theorem 2.17]{ima:tlpaou}}]\label{cgutheo}
Let $\mg$ be an ultragraph in which every cycle in $\mg$ has an exit, $K$ a field, $A$ a $K$-algebra, and $\vp:L_K(\mg)\to A$ a $K$-algebra homomorphism. Then, $\vp$ is injective if and only if $\vp(p_v)\neq 0$ for all $v\in G^0$.
\end{thm}
\begin{proof}
($\Longrightarrow$) Assume that $\vp$ is injective. By Lemma \ref{graded} (1), we have $p_v\neq 0$ for all $v\in G^0$,  and so $\vp(p_v)\neq 0$ for all $v\in G^0$.

($\Longleftarrow$) Assume that $\vp(p_v)\neq 0$ for all $v\in G^0$, and that $\ker(\vp)\neq 0$. By Corollary~ \ref{gi-cor} (3), $\ker(\vp)$ contains an element $p_v$, for some $v\in G^0$. This implies that $\vp(p_v)=0$, a contradiction. Thus, $\ker(\vp)= \{0\}$ and $\vp$ is injective, as desired.	
\end{proof}

Recall that a ring $R$ is said to be {\it semiprime} if, for every ideal $I$ of $R$, $I^2 = 0$ implies $I = 0$. In \cite[Corollary 3.3]{gon:ratrt19}, as an application of the reduction theorem, Royer and the second author showed that ultragraph Leavitt path algebras are semiprime. Next, we recover this result using Theorem \ref{gi-generatingset}.

\begin{thm}[{cf. \cite[Corollary 3.3]{gon:ratrt19}}]\label{semiprime}
Let $\mg$ be an ultragraph and $K$ a field. Then the ultragraph Leavitt path algebra $L_K(\mg)$ is semiprime.	
\end{thm}
\begin{proof}
By Lemma \ref{graded}, $L_K(\mg)$ is $\mathbb{Z}$-graded. Then, by \cite[Proposition II.1.4 (1)]{NO:1982}, it suffices to check that the only graded ideal $I$ of $L_K(\mg)$ for which $I^2=0$ is $I=0$.
By Theorem \ref{gi-generatingset} (2), any graded ideal of $L_K(\mg)$ is generated by idempotents, and so the result follows immediately (see also Corollary~\ref{gi-cor}).
\end{proof}

We now establish the semiprimitivity result.

\begin{thm}\label{semiprimitive}
Let K be a field and $\mg$ an ultragraph. Then  the ultragraph Leavitt path algebra $L_K(\mg)$ is semiprimitive, i.e, $J(L_K(\mg))=0$.
\end{thm}
\begin{proof}
It is well-known (see, e.g., \cite[Lemma 2.5]{NN:2020}) that $L_K(\mg)$ is an algebra with
local units (specifically, the set of local units of $L_K(\mg)$ is given by $\{p_A\mid A\in \mg^0\}$). Then, by \cite[Lemma 6.2]{ap:tlpaoag08}, $J(L_K(\mg))$ is a graded ideal of $L_K(\mg)$. But, by Theorem~\ref{gi-generatingset} (2), $J(L_K(\mg))$ is generated by idempotents. Since the Jacobson radical of any ring contains no nonzero idempotents, we conclude $J(L_K(\mg))=0$, thus finishing the proof.
\end{proof}

\section{Prime ideals of ultragraph Leavitt path algebras}
The main aim of this section is to give a complete characterization of the prime ideals of an ultragraph Leavitt path algebra (Theorem \ref{prime6}). Consequently, we provide a method of constructing non-graded prime ideals of ultragraph Leavitt path algebras (Corollary \ref{prime8}).

We start extending the definition of a downward direct graph to ultragraphs.

\begin{defn}
For any ultragraph $\mg$ and vertices $v, w\in G^0$, we write $v \ge w$ if there exists a  path $p\in \mathcal{G}^*$ such that $s(p) = v$ and $w\in r(p)$. An ultragraph $\mg$ is called {\it downward directed} if, for any two $v, w\in G^0$ there exists a vertex $u\in G^0$ such that $v\ge u$ and $w\ge u$.
\end{defn}

In \cite[Theorem 2.4]{abr:opnpvnra}, Abrams, Bell, and Rangaswamy give a criterion for the Leavitt path algebras of an arbitrary graph to be prime. We extend this result to ultragraph Leavitt path algebras below.

\begin{thm}\label{prime1}
Let $K$ be an ultragraph and $K$ a field. Then, $L_K(\mg)$ is a prime ring if and only if $\mg$ is downward directed.
\end{thm}
\begin{proof}
($\Longrightarrow$) Assume that $L_K(\mg)$ is a prime ring. Let $v, w\in G^0$. Since the ideals $L_K(\mg)p_vL_K(\mg)$ and $L_K(\mg)p_wL_K(\mg)$ are nonzero, $L_K(\mg)p_vL_K(\mg)p_wL_K(\mg)$ is nonzero, and so $p_vL_K(\mg)p_w$  is nonzero. This implies that $p_vs_{\alpha}p_As^*_{\beta}p_w \neq 0$, for some $A\in \mg^0$ and $\alpha, \beta\in\mg^*$ with $r(\alpha)\cap A\cap r(\beta) \neq \emptyset$. Consider the following cases:
	
{\it Case 1}: $|\alpha| \ge 1$ and $|\beta| \ge 1$. We then have that
$s(\alpha) = v$ and $s(\beta) = w$, and so $v\ge u$ and $w\ge u$ for all $u\in r(\alpha)\cap A\cap r(\beta)$, as desired. 
	
{\it Case 2}: $|\alpha| = 0$ and $|\beta| \ge 1$. We then have that $s(\beta) = w$, $s_{\alpha}= p_B$ for some $B\in \mg^0$, and $v\in B\cap A\cap r(\beta)$, and hence $w\ge v$, as desired.
	
{\it Case 3}: $|\alpha| \ge 1$ and $|\beta| =0$. We then have that $s(\alpha) = v$, $s_{\beta}= p_C$ for some $C\in \mg^0$, and $w\in r(\alpha) \cap A\cap C$, and hence $v\ge w$, as desired.
	
In any of the above cases, we obtain that $v\ge u$ and $w\ge u$ for some $u\in G^0$. Thus, $\mg$ is downward directed.
	
($\Longleftarrow$) Assume that $\mg$ is downward directed. By Lemma \ref{graded}, $L_K(\mg)$ is a $\mathbb{Z}$-graded $K$-algebra. Then, by \cite[Proposition II.1.4]{NO:1982}, to establish the primeness of $L_K(\mg)$, we only need to show that $IJ \neq 0$ for any pair $I, J$ of nonzero graded ideals of $L_K(\mg)$. So, let $I$ and $J$ be nonzero graded ideals of  $L_K(\mg)$. By Corollary \ref{gi-cor} (2), there exist two vertices $v, w\in G^0$ such that $p_v\in I$ and $p_w\in J$. By downward directedness there exists a vertex $u\in G^0$ such that $v\ge u$ and $w\ge u$, that is, there exists two paths $\alpha, \beta\in \mg^*$ such that $s(\alpha) = v$, $s(\beta) = w$ and $u\in r(\alpha)\cap r(\beta)$. This implies that $p_u = p_u s^*_{\alpha}p_v s_{\alpha}\in I$ and $p_u = p_u s^*_{\beta}p_v s_{\beta}\in J$, and so $0\neq p_u= p_up_u\in IJ$, as desired.
\end{proof}
 
Next, we present a series of lemmas, which will lead to the characterization of the prime ideals in ultragraph Leavitt path algebras.

\begin{lem}\label{prime2}
Let $K$ be a field, $\mg$ an ultragraph and $I$ a nonzero ideal of $L_K(\mg)$ which does not contain $p_v$ for all $v\in G^0$. Then, $I$ is a non-graded ideal generated by elements of the form $x = p_v + \sum^n_{i=1}k_is^{r_i}_c$, where $c$ is a unique cycle without exits based at $v$, and  $k_i\in K$ with at least one $k_i\neq 0$.	
\end{lem}
\begin{proof}
Since $I$ is a nonzero ideal which does not contain $p_v$ for all $v\in G^0$, we obtain, by Corollary~\ref{gi-cor} (2), that $I$ is a non-graded ideal.
By Theorem \ref{generatingset}, $I$ is generated as an ideal by elements in $I$ of the form	
\[x= (p_A-\sum_{i=1}^nk_is_c^{r_i})(p_A-\sum_{e \in S}s_es^*_e)\neq 0,\]	
where $A\in \mathcal{G}^0$, $k_1, \hdots, k_n\in K$, $r_1, \hdots, r_n$ are positive integers, $S$ is a finite subset of $\mathcal{G}^1$ consisting of edges with the same source vertex $v\in A$, and, whenever $k_i\neq 0$ for some $1\le i\le n$, $c$ is the unique cycle based at $v$ such that $A\subseteq r(c)$.

Let $x$ be a generating element of $I$ as above. If $v$ is a sink, then $x=p_v \in I$, a contradiction. So $v$ is not a sink. Suppose that $s^{-1}(v)\setminus S=\emptyset$. If $A=\{v\}$, then $x=0$, a contradiction. If $A\neq\{v\}$, then there exists $w\in A\setminus\{v\}$ and so $p_w=p_wx\in I$, a contradiction. Hence, we must have $s^{-1}(v)\setminus S\neq \emptyset$, that is,  there exists an edge $f\in s^{-1}(v)\setminus S$.

If $f$ is not the initial edge of $c$, then $s^*_fs_c=0$ and $$p_{r(f)} = s^*_f(p_A-\sum_{i=1}^nk_is_c^{r_i})(p_A-\sum_{e \in S}s_es^*_e) s_f= s^*_fxs_f\in I,$$ so $p_v\in I$ for all $v\in r(f)$, a contradiction. 
Therefore, $f$ must be the initial edge of $c$, say $c = fg$. Then, $$y:=s^*_fxs_f = s^*_f(p_A-\sum_{i=1}^nk_is_c^{r_i})(p_A-\sum_{e \in S}s_es^*_e) s_f = p_{r(f)} + \sum^n_{i=1}k_is^{r_i}_h\in I,$$ where $h$ is the cycle $gf$. If $c$ has an exit, then $h$ also has an exit. Write $h= h_1\cdots h_m$. So, there is either an edge $e\in \mathcal{G}^1$ such that there exists an $1\le i\le m$ for which $s(e)\in r(h_i)$ but $e\ne h_{i+1}$, or a sink $w$ such that $w\in r(h_i)$ for some $1\le i\le m$.  
In the first case, for $t:= h_1 \cdots h_ie$, we have that
$s^*_{t}ys_{t}=s^*_{t}p_{r(f)}s_t=p_{r(t)}\in I,$
so $p_u \in I$ for all $u\in r(h)$, a contradiction. In the second case, for $t:= h_1 \cdots h_iw$, we have that $s^*_{t}ys_{t}=p_w\in I$,  a contradiction again.

Thus, $c$ has no exits, and so we obtain that $r(c) = \{v\} = A$ and $|s^{-1}(v)| =1$. This implies that $S$ is the empty set (since $x\neq 0$). Therefore, the generators of $I$ are of the form $x = p_v + \sum^n_{i=1}k_is^{r_i}_c$, where $c$ is a unique cycle without exits based at $v$  and $k_i\in K$ with at least one $k_i\neq 0$, thus the proof is finished.
\end{proof}

The proofs of the next three lemmas are straightforward adaptations of the proofs of their corresponding graph versions given in \cite[Lemma 3.5]{Ran:ttopiolpaoag}. We include it here for completeness.

\begin{lem}[{cf.  \cite[Lemma 3.5]{Ran:ttopiolpaoag}}]\label{prime3}
Let $K$ be a field, $\mg$ a downward directed ultragraph, and $I$ a nonzero ideal of $L_K(\mg)$ which does not contain $p_v$ for all $v\in G^0$. Then, there is a unique cycle $c$ without exits in $\mg$, and $I$ is a non-graded principal ideal generated by $p(s_c)$, where $p(x)$ is a polynomial in $K[x]$.	
\end{lem}
\begin{proof}  By Lemma \ref{prime2}, $I$ is a non-graded ideal generated by elements of the form $p_w + \sum^n_{i=1}k_is^{r_i}_g$, where $g$ is a unique cycle without exits based at $w$ and  $k_i\in K$ with at least one $k_i\neq 0$.	Since $\mg$ is downward directed, $g$ is the only cycle without exits in $\mg$, except possibly a permutation of its vertices. This implies that if there is a cycle $c$ without exits based at a vertex $v$ in $\mg$, then the cycle $g$ based at $w$ is the same as the cycle $c$ based at $v$, obtained possibly by a rotation of the vertices on $c$. We note that if $\alpha$ is the part of $c$ from $v$ to $w$ and $\beta$ is the part of $c$ from $w$ to $v$, then $s^*_{\beta}(p_w + \sum^n_{i=1}k_is^{r_i}_g)s_{\beta}=p_v + \sum^n_{i=1}k_is^{r_i}_c$ and $s^*_{\alpha}(p_v + \sum^n_{i=1}k_is^{r_i}_c)s_{\alpha}=p_w + \sum^n_{i=1}k_is^{r_i}_g$. This implies that we may select a generating set for $I$ consisting of elements of the form $p_v + \sum^n_{i=1}k_is^{r_i}_c$, with the fixed cycle $c$ based at $v$, and where $k_i\neq 0$ for at least one $i$. Let $f(s_c) := p_v + \sum^n_{i=1}k_is^{r_i}_c$, where $f(x) = x^0 + \sum^n_{i=1}k_ix^{r_i}\in K[x]$ and we use the convention that $s^0_c = p_v$. Let $p(x)$ be a polynomial with the smallest positive degree in $K[x]$ such that $p(s_c)\in I$. Then, by the division algorithm in $K[x]$, every generator $f(s_c)$ of $I$ is a multiple of $p(s_c)$, and so $I$ is the principal ideal generated by $p(s_c)$, and the proof is finished.
\end{proof}	

Recall the notion of an admissible pair $(\mathcal{H}, S)$ of $\mg$ (see Section~\ref{sec:appl}). We denote the ideal of $L_K(\mg)$ generated by the set $\{p_A\mid A\in \mathcal{H}\} \cup \{p^{\mathcal{H}}_v\mid v\in S\}$ by $I_{(\mathcal{H}, S)}$.

\begin{lem}[{cf.  \cite[Lemma 3.6]{Ran:ttopiolpaoag}}]\label{prime4}
Let $K$ be a field, $\mg$ an ultragraph, and $I$ an ideal of $L_K(\mg)$. Then, the graded ideal $I_{(\mathcal{H}_I, S_I)} \subseteq I$ contains every other graded ideal of $L_K(\mg)$ inside $I$.
\end{lem}
\begin{proof}
Let $P$ be a graded ideal of $L_K(\mg)$ contained in $I$. By Theorem \ref{gi-generatingset} (3) (see, also \cite[Theorem 3.4]{ima:tlpaou}), $P = I_{(\mathcal{H}_P, S_P)}$.	We claim that $P \subseteq I_{(\mathcal{H}_I, S_I)}$. It is obvious that $\mathcal{H}_P \subseteq \mathcal{H}_I$. If $\mathcal{H}_P = \emptyset$, then $P = 0 \subseteq \mathcal{H}_I$. Consider the case when $\mathcal{H}_P \neq \emptyset$. We need to show that $p^{\mathcal{H}_P}_v \in I_{(\mathcal{H}_I, S_I)}$ for all $v\in S_P$. Let $v\in S_P$ and let $e_1, \hdots, e_n$ be the edges having $s(e_i) = v$ and $r(e_i)\notin \mathcal{H}_P$. 

Suppose that $v$ is a breaking vertex for $\mathcal{H}_I$. By re-indexing, we may assume that for some $m\le n$, $r(e_i) \notin \mathcal{H}_I$ for all $1\le i\le m$, and that $r(e_j) \in \mathcal{H}_I$ for all $m+1\le j\le n$. Since $s_{e_j}\in I_{(\mathcal{H}_I, S_I)}$ for all
$m+1\le j\le n$, we have $p^{\mathcal{H}_P}_v= p^{\mathcal{H}_I}_v-\sum^n_{j=m+1}s_{e_j}s^*_{e_j}\in I_{(\mathcal{H}_I, S_I)}$, as desired. 

Suppose that $v$ is not a breaking vertex for $\mathcal{H}_I$. Then, since $v$ is a breaking vertex for $\mathcal{H}_P$, $r(s^{-1}(v))\subseteq \mathcal{H}_I$, and so $s_{e_i}= s_{e_i}p_{r(e_i)}\in I_{(\mathcal{H}_I, S_I)}$ for all $1 \le i\le n$. This implies that $\sum^n_{i=1}s_{e_i}s^*_{e_i}\in I_{(\mathcal{H}_I, S_I)} \subseteq I$. Since $p_v - \sum^n_{i=1}s_{e_i}s^*_{e_i}\in I$, we must have $p_v\in I$, that is, $v\in \mathcal{H}_I$. Then, clearly we have that 
$p^{\mathcal{H}_P}_v\in I_{(\mathcal{H}_I, S_I)}$, what shows the claim and finishes the proof.
\end{proof}

\begin{lem}[{cf. \cite[Lemma 3.8]{Ran:ttopiolpaoag}}]\label{prime5}
Let $K$ be a field, $\mg$ an ultragraph, and $I$ a prime ideal of $L_K(\mg)$. Then, $I_{(\mathcal{H}_I, S_I)}$ is a prime ideal of $L_K(\mg)$.
\end{lem}
\begin{proof}
Since $L_K(\mg)$ is $\mathbb{Z}$-graded, by \cite[Proposition II.1.4]{NO:1982}, $I_{(\mathcal{H}_I, S_I)}$ is a prime ideal if, and only if, $I_{(\mathcal{H}_I, S_I)}$ is a  graded prime ideal.

Let $A = I_{(\mathcal{H}_1, S_1)}$ and $B= I_{(\mathcal{H}_2, S_2)}$ be graded ideals of $L_K(\mg)$ with $AB\subseteq I_{(\mathcal{H}_I, S_I)}$. Since $I$ is prime, one of them, say $A$, is contained in $I$. By Lemma \ref{prime4}, $A\subseteq I_{(\mathcal{H}_I, S_I)}$, and so $I_{(\mathcal{H}_I, S_I)}$ is graded prime, and the proof is finished. 
\end{proof}

To establish the description of prime ideals, we need first to recall the notion of the quotient ultragraph of an ultragraph, which was introduced in \cite{Larki:piapioua19} and \cite{ima:tlpaou}. Let $\mg=(G^0, \mg^1, r_{\mg}, s_{\mg})$ be an ultragraph and $(\mathcal{H}, S)$ an admissible pair in $\mg$. For each $A\in \mg^0$, we let $\overline{A}:= A \cup \{w'\mid w\in A \cap (B_{\mathcal{H}}\setminus S)\}$, where $w'$ is another copy of $w$. Define an ultragraph $\overline{\mg}= (\overline{G}^0, \overline{\mg}^1, \overline{r}, \overline{s})$, where $\overline{G}^0 = G^0 \cup \{w'\mid w\in B_{\mathcal{H}}\setminus S\}$, $\overline{\mg}^1 = \mg^1$, $\overline{r}(e) = \overline{r_{\mg}(e)}$ and 
\[ \overline{s}(e) = \begin{cases}
s_{\mg}(e)' & \text{if } s_{\mg}(e) \in B_{\mathcal{H}}\setminus S \text{ and } r_{\mg}(e) \in \mathcal{H} , \\
s_{\mg}(e) & \text{otherwise} \end{cases} \] for all $e\in \overline{\mg}^1$. We denote by $\overline{\mg}^0$ the smallest subset of $\mathcal{P}(\overline{G}^0)$ that contains $\{v\}$ for all $v\in \overline{G}^0$,  contains $\overline{r}(e)$ for all $e\in \overline{\mg}^1$, and is closed under relative complements, finite unions and finite intersections. We note that $\overline{A} = A$ for all $A\in \mathcal{H}$, and so $\mathcal{H}$ is a saturated hereditary subset of $\overline{G}^0$ and the set of breaking
vertices of $\mathcal{H}$ in $\overline{\mg}$ is $S$. Moreover, by \cite[Lemma 3.3]{ima:tlpaou}, $L_K(\mg)$ is isomorphic to $L_K(\overline{\mg})$ as $\mathbb{Z}$-graded $K$-algebras.

By \cite[Lemma 2.3]{ima:tlpaou} we have an equivalent relation $\sim$ on $\overline{\mg}^0$ defined by $A \sim B$ if and only if $A \cap V = B\cap V$ for some $V\in \mathcal{H}$, and the operations
\begin{center}
$[A] \cap [B]:= [A \cap B]$, $[A] \cup [B]:= [A \cup B]$ and $[A] \setminus [B]:= [A \setminus B]$	
\end{center}
are well-defined on the equivalence classes $\{[A]\mid A\in \overline{\mg}^0\}$. We usually denote $[v]$ instead of $[\{v\}]$ for all $v\in \overline{G}^0$, and the set $\bigcup_{A\in \mathcal{H}}A$ is denoted by $\bigcup \mathcal{H}$.

The {\it quotient ultragraph} of $\mathcal{G}$ by $(\mathcal{H}, S)$ is the quadruple $$\mg/(\mathcal{H}, S) = (\Phi(G^0), \Phi(\mg^1), r, s),$$ where $$\Phi(G^0) := \{[v]\mid v\in G^0\setminus \bigcup \mathcal{H}\} \cup \{[w']\mid w\in B_{\mathcal{H}}\setminus S\},$$
$$\Phi(\mg^1) := \{e\in \mg^1\mid r_{\mg}(e) \notin\mathcal{H}\},$$
and $s: \Phi(\mg^1)\longrightarrow \Phi(G^0)$ and $r: \Phi(\mg^1)\longrightarrow \{[A]\mid A\in \overline{\mg}^0\}$ are the maps defined respectively by $s(e) = [s_{\mg}(e)]$ and $r(e) = [\overline{r_{\mg}(e)}]$ for all $e\in \Phi(\mg^1)$.

For any ultragraph $\mg$, and any vertex $v\in G^0$, we define $M(v) = \{w\in G^0\mid w\ge v\}$. 
For any path $\alpha = \alpha_1\alpha_2\cdots \alpha_n$ in $\mg$, the set $\{v\in G^0\mid v= s(\alpha_i) \text{ for some } 1\le i\le n\}$ is called the {\it set of all vertices} on $\alpha$. Following \cite{ar:fpsmolpa}, a cycle $c$ in $\mg$ is called an {\it exclusive cycle} if it is disjoint with every other cycle; equivalently, no vertex on $c$ is the base of a different cycle other than the cyclic permutation of $c$.

We are now in a position to provide the main result of this section, which extends Rangaswamy's result \cite[Theorem 3.13]{Ran:ttopiolpaoag} to ultragraph Leavitt path algebras.

\begin{thm}\label{prime6}
Let $K$ be a field, $\mg$ an ultragraph, and $P$ an ideal of $L_K(\mg)$ with $\mathcal{H} = \{A\in \mg^0\mid p_A \in P\}$. Then, $P$ is a prime ideal of $L_K(\mg)$ if, and only if, $P$ satisfies one of the following conditions: 
\begin{itemize} 	
\item[(i)] $P$ is generated by the set $\{p_A\mid A\in \mathcal{H}\} \cup \{ p^{\mathcal{H}}_v\mid v\in  B_{\mathcal{H}}\}$ and $G^0\setminus \bigcup \mathcal{H}$ is downward directed;
\item[(ii)] $P$ is generated by the set $\{p_A\mid A\in \mathcal{H}\} \cup \{ p^{\mathcal{H}}_v\mid v\in B_{\mathcal{H}}\setminus\{w\}\}$
for some $w\in B_{\mathcal{H}}$ and $G^0\setminus \bigcup \mathcal{H} = M(w)$;
\item[(iii)] $P$ is generated by the set $\{p_A\mid A\in \mathcal{H}\} \cup \{ p^{\mathcal{H}}_v\mid v\in  B_{\mathcal{H}}\} \cup \{f(s_c)\}$, where $c$ is an exclusive cycle in $\mg$ based at a vertex $v$, $G^0\setminus \bigcup \mathcal{H} = M(v)$, and $f(x)$ is an irreducible polynomial in $K[x, x^{-1}]$.
\end{itemize}
\end{thm}	
\begin{proof}
We let $S:= \{w\in B_{\mathcal{H}}\mid p^{\mathcal{H}}_v\in P\}$. Consider the following two cases:

{\it Case $1$}: $P$ is a graded ideal. Then, $P = I_{(\mathcal{H}, S)}$ (by Theorem \ref{gi-generatingset} (3) or \cite[Theorem 3.4 (2)]{ima:tlpaou}) and $L_K(\mg)/P \cong L_K(\mg/(\mathcal{H}, S))$ as $\mathbb{Z}$-graded $K$-algebras (by \cite[Theorem 3.4 (1)]{ima:tlpaou}). 
%Let $\phi: L_K(\mg)/P \longrightarrow L_K(\mg/(\mathcal{H}, S))$ be an isomorphism of $\mathbb{Z}$-graded $K$-algebras. 
Therefore, $P$ is prime if and only if $L_K(\mg/(\mathcal{H}, S))$ is a prime ring. By Theorem~\ref{prime1}, this is equivalent to $\mg/(\mathcal{H}, S)$ be downward directed. Now, for every $w\in B_{\mathcal{H}}\setminus S$, the corresponding vertex $[w']$ is a sink in the ultragraph $\mg/(\mathcal{H}, S)$. Since $\mg/(\mathcal{H}, S)$ is downward directed, there is at most one sink in $\mg/(\mathcal{H}, S)$, and so $B_{\mathcal{H}}\setminus S$ is either empty or a singleton $\{w\}$. Thus, $P$ is prime if and only if $B_{\mathcal{H}} = S$, in which case $G^0\setminus \bigcup \mathcal{H}$ is downward directed, or $B_{\mathcal{H}}\setminus \{w\} = S$, in which case 
$\Phi(G^0) = \{[v]\mid v\in G^0\setminus \bigcup \mathcal{H}\} \cup \{[w']\}$ and $[v] \ge [w']$ for all $[v]\in \Phi(G^0)$ (equivalently, $v\ge w$ for all $v\in G^0\setminus \bigcup \mathcal{H}$). Thus, the primeness of the graded ideal $P$ is equivalent to either Condition~(i) or Condition~(ii).

{\it Case $2$}: $P$ is a non-graded ideal. Assume that $P$ is prime. By Lemma \ref{prime5}, $I_{(\mathcal{H}, S)}$ is a graded prime ideal of $L_K(\mg)$ contained in $P$, and so, as showed in Case~$1$, either: 
\begin{enumerate}
    \item $B_{\mathcal{H}} = S$ and $G^0\setminus \bigcup \mathcal{H}$ is downward directed, or;
    \item $B_{\mathcal{H}}\setminus \{w\} = S$ and $v\ge w$ for all $v\in G^0\setminus \bigcup \mathcal{H}$.
\end{enumerate}

By \cite[Theorem 3.4 (1)]{ima:tlpaou}, $L_K(\mg)/I_{(\mathcal{H}, S)} \cong L_K(\mg/(\mathcal{H}, S))$ as $\mathbb{Z}$-graded $K$-algebras. 
Let $\phi: L_K(\mg)/I_{(\mathcal{H}, S)} \longrightarrow L_K(\mg/(\mathcal{H}, S))$ be an isomorphism of $\mathbb{Z}$-graded $K$-algebras and $Q := \phi(P/I_{(\mathcal{H}, S)})$. 

We note that, under condition (1), $\Phi(G^0) = \{[v]\mid v\in G^0\setminus \bigcup \mathcal{H}\}$ and, under condition (2), $\Phi(G^0) = \{[v]\mid v\in G^0\setminus \bigcup \mathcal{H}\} \cup \{[w']\}$ and  $[v] \ge [w']$ for all $[v]\in \Phi(G^0)$, where $p_{[w']} = \phi(p^{\mathcal{H}}_w + I_{(\mathcal{H}, S)})$. Since $w\notin S$, we have that $p^{\mathcal{H}}_w \notin P$, and so $p_{[w']}\notin Q$. Thus, in either case, we always get that $Q$ is a non-zero ideal of
$L_K(\mg/(\mathcal{H}, S))$, which does not contain $p_{[v]}$ for all $[v]\in \Phi(G^0)$. Since $\mg/(\mathcal{H}, S)$ is downward directed, by Lemma~\ref{prime3}, there is a unique cycle $c$ without exits in $\mg/(\mathcal{H}, S)$ such that $Q$ is a non-graded principal ideal generated by $p(s_c)$, where $p(x)$ is a polynomial in $K[x]$. We claim that $B_{\mathcal{H}} = S$. Indeed, suppose that $B_{\mathcal{H}}\setminus \{w\} = S$. Then, $[v] \ge [w']$ for all $[v]\in \Phi(G^0)$. Since $c$ has no exits, $[w']$ must lie on $c$ and hence $[w']$ is not a sink in $\mg/(\mathcal{H}, S)$, a contradiction, showing the claim. So, we have that $P$ is generated by $\{p_A\mid A\in \mathcal{H}\} \cup \{ p^{\mathcal{H}}_v\mid v\in  B_{\mathcal{H}}\} \cup \{f(s_c)\}$. Since $\Phi(G^0) = \{[v]\mid v\in G^0\setminus \bigcup \mathcal{H}\}$ is downward directed, and contains all vertices which lie on $c$, we obtain that $G^0\setminus \bigcup \mathcal{H} = M(v)$, where $v$ is the base of the cycle $c$. It is also clear that  no vertex on $c$ is the base of another distinct cycle in $\mg$.

We show that $f(x)$ is irreducible. Since $Q$ is a prime ideal of 
$L_K(\mg/(\mathcal{H}, S))$, by \cite[Lemma 3.10]{Ran:ttopiolpaoag}, $p_{[v]}Qp_{[v]}$ is a nonzero prime ideal of $p_{[v]}L_K(\mg/(\mathcal{H}, S))p_{[v]}$ generated by $p_{[v]} f(s_c)p_{[v]} = f(s_c)$. Since $c$ is a cycle without exits in $\mg/(\mathcal{H}, S)$, by Lemma~\ref{graded} (2), we have that $p_{[v]}L_K(\mg/(\mathcal{H}, S))p_{[v]}\cong K[x, x^{-1}]$ via the isomorphism $\lambda$ defined by: $p_{[v]}\longmapsto 1$, $s_{c}\longmapsto x$ and $s^*_{c}\longmapsto x^{-1}$. We should mention that the isomorphism $\lambda$ maps $f(s_c)$ to $f(x)$. Since $f(x)$ generates the non-zero prime ideal $\lambda(p_{[v]}Qp_{[v]})$ in $K[x, x^{-1}]$, $f(x)$ is an irreducible polynomial in $K[x, x^{-1}]$, as desired. 

Conversely, assume that $\mg$ contains a cycle $c$ based at a vertex $v$ such that  no vertex on $c$ is the base of another distinct cycle in $\mg$, $G^0\setminus \bigcup \mathcal{H} = M(v)$ and  $P$ is generated by the set $\{p_A\mid A\in \mathcal{H}\} \cup \{ p^{\mathcal{H}}_v\mid v\in  B_{\mathcal{H}}\} \cup \{f(s_c)\}$ for some irreducible polynomial $f(x)\in K[x, x^{-1}]$.
We then have that the quotient ultragraph $\mg/(\mathcal{H}, B_{\mathcal{H}})$ is downward directed and contains the cycle $c$ without exits. By \cite[Theorem 3.4 (1)]{ima:tlpaou}, there exists a graded isomorphism $\phi: L_K(\mg)/I_{(\mathcal{H}, B_{\mathcal{H}})} \longrightarrow L_K(\mg/(\mathcal{H}, B_{\mathcal{H}}))$. Let $Q := \phi(P/I_{(\mathcal{H}, B_{\mathcal{H}})})$. Since $P$ is generated by the set $\{p_A\mid A\in \mathcal{H}\} \cup \{ p^{\mathcal{H}}_v\mid v\in  B_{\mathcal{H}}\} \cup \{f(s_c)\}$, $Q$ is generated by $f(s_c)$. Since $f(x)$ is irreducible in $K[x, x^{-1}]$, and $p_{[v]}L_K(\mg/(\mathcal{H}, B_{\mathcal{H}}))p_{[v]}\cong K[x, x^{-1}]$ via the isomorphism $\lambda$ defined by: $p_{[v]}\longmapsto 1$, $s_{c}\longmapsto x$ and $s^*_{c}\longmapsto x^{-1}$, we obtain that the ideal $p_{[v]}Q p_{[v]}$, being generated by $p_{[v]} f(s_c) p_{[v]} = f(s_c) = \lambda^{-1}(f(x))$, is a maximal ideal of the $K$-algebra $p_{[v]}L_K(\mg/(\mathcal{H}, B_{\mathcal{H}}))p_{[v]}$.

We claim that $Q$ is a prime ideal of $L_K(\mg/(\mathcal{H}, B_{\mathcal{H}}))$. Indeed, let $A$ and $B$ be ideals of $L_K(\mg/(\mathcal{H}, B_{\mathcal{H}}))$ such that $AB\subseteq Q$. We then have $p_{[v]}Ap_{[v]}p_{[v]}Bp_{[v]}\subseteq p_{[v]}Qp_{[v]}$, and so one of them, say $p_{[v]}Ap_{[v]}\subseteq p_{[v]}Qp_{[v]}$. If $A$ contains an element $p_{[w]}$ for some  vertex $[w]$ in $\mg/(\mathcal{H}, B_{\mathcal{H}})$, then there exists a path $p$ in $\mg/(\mathcal{H}, B_{\mathcal{H}})$ such that $s(p)= [w]$ and $[v]\in r(p)$ (since $w\ge v$), and so $$p_{[v]} = p_{[v]}s^*_p p_{[w]}s_pp_{[v]} \in A$$ and $p_{[v]}L_K(\mg/(\mathcal{H}, B_{\mathcal{H}}))p_{[v]}\subseteq A$. This implies that $p_{[v]}L_K(\mg/(\mathcal{H}, B_{\mathcal{H}}))p_{[v]}\subseteq p_{[v]}Ap_{[v]} \subseteq  p_{[v]}Qp_{[v]}$, a contradiction to the fact that $p_{[v]}Qp_{[v]}$ is a proper ideal of $p_{[v]}L_K(\mg/(\mathcal{H}, B_{\mathcal{H}}))p_{[v]}$. Hence, $A$ does not contain  $p_{[w]}$ for every  vertex $[w]$ in $\mg/(\mathcal{H}, B_{\mathcal{H}})$. By Lemma \ref{prime4}, the ideal $A$ of $L_K(\mg/(\mathcal{H}, B_{\mathcal{H}}))$  is generated by $q(s_c)$ for some polynomial $q(x)\in K[x]$. Since $p_{[v]}q(s_c) p_{[v]}= q(s_c)\in p_{[v]}Ap_{[v]} \subseteq p_{[v]}Q p_{[v]} \subseteq Q$, we get that $A \subseteq Q$, and so $Q$ is prime, showing the claim. This implies that $P$ is a prime ideal of $L_K(\mg)$.

Suppose $P$ is a graded ideal. Then, by Lemma \ref{prime4}, $P = I_{(\mathcal{H}, B_{\mathcal{H}})}$, and so $Q = 0$. On the other hand, $p_{[v]}Q p_{[v]}$ is generated by $p_{[v]} f(s_c) p_{[v]} = f(s_c) = \lambda^{-1}(f(x))\neq 0$, and so $Q \neq 0$, a contradiction. This shows that $P$ is a non-graded prime ideal of $L_K(\mg)$, thus finishing the proof.
\end{proof}

In \cite[Theorem 4.3]{ima:tlpaou} the authors showed that an ultragraph $\mg$ satisfies  Condition~$(K)$ if and only if every ideal of $L_K(\mg)$ is graded. As a corollary of Theorem~\ref{prime6}, we have the following.
	
\begin{cor}\label{prime7}
Let $K$ be a field and $\mg$ an ultragraph. Then $\mg$ satisfies Condition~$(K)$ if, and only if, every prime ideal of $L_K(\mg)$ is graded.
\end{cor}
\begin{proof}
($\Longrightarrow$)	It immediately follows from Theorem \ref{prime6}.

($\Longleftarrow$) Assume that every prime ideal of $L_K(\mg)$ is graded and $\mg$ does not satisfy the Condition~$(K)$. Then, there exists a cycle $c$ in $\mg$, based at a vertex $v$, such that no vertex on $c$ is the base of another distinct cycle in $\mg$. Let $\mathcal{H} =\{A\in \mg^0\mid w\ngeq v \text{ for all } w\in A\}$. We note that $\mathcal{H}$ does not contain $\{u\}$ for all vertex $u$ on $c$. 

We claim that $\mathcal{H}$ is a hereditary and saturated subset of $\mg^0$. Indeed, let $e$ be an edge in $\mg$ with $\{s(e)\}\in \mathcal{H}$. If $w\ge v$ for some $w\in r(e)$, then $s(e)\ge v$, i.e., $\{s(e)\}\notin \mathcal{H}$, a contradiction, and so $r(e)\in \mathcal{H}$. It is obvious that $A \cap B \in \mathcal{H}$ for all $A, B\in \mathcal{H}$. Furthermore, it is also clear that if $B\in \mg^0$ with $B\subseteq A$ and $A\in \mathcal{H}$, then $B\in \mathcal{H}$. These observations show that $\mathcal{H}$ is hereditary. Let $w$ be a regular vertex in $\mg$ such that $r(e)\in \mathcal{H}$ for all $e\in s^{-1}(w)$. If $\{w\}\notin \mathcal{H}$, then $w\ge v$, and so there exists a path $p = e_1e_2\cdots e_m$ in $\mg$ such that $w = s(p)$ and $v\in r(p)$. This implies that $s(e_2)\ge v$, and so $r(e_1)\notin \mathcal{H}$. On the other hand, since $e_1\in s^{-1}(w)$, we must have $r(e_1)\in \mathcal{H}$, a contradiction. Thus, $\mathcal{H}$ is a hereditary and saturated subset of $\mg^0$, and the claim is proved.

From the construction of $\mathcal{H}$, we obtain that $c$ is a cycle without exits and based at $[v]$ in the quotient ultragraph $\mg/(\mathcal{H}, B_{\mathcal{H}})$. So, $p_{[v]}L_K(\mg/(\mathcal{H}, B_{\mathcal{H}}))p_{[v]}\cong K[x, x^{-1}]$ via the isomorphism defined by: $p_{[v]}\longmapsto 1$, $s_{c}\longmapsto x$ and $s^*_{c}\longmapsto x^{-1}$. Let $f(x)$ be an irreducible polynomial in $K[x, x^{-1}]$.
By \cite[Theorem 3.4 (1)]{ima:tlpaou}, there exists a $K$-algebra graded isomorphism $\phi: L_K(\mg)/I_{(\mathcal{H}, B_{\mathcal{H}})} \longrightarrow L_K(\mg/(\mathcal{H}, B_{\mathcal{H}}))$.  Let $P$ be an ideal of $L_K(\mg)$ such that $I_{(\mathcal{H}, B_{\mathcal{H}})} \subseteq P$ and the ideal $\phi(P/I_{(\mathcal{H}, B_{\mathcal{H}})})$ of $L_K(\mg/(\mathcal{H}, B_{\mathcal{H}}))$ is generated by $f(s_c)$. Then, $P$ is generated by the set $\{p_A\mid A\in \mathcal{H}\} \cup \{ p^{\mathcal{H}}_v\mid v\in  B_{\mathcal{H}}\} \cup \{f(s_c)\}$ and so, by Theorem~\ref{prime6}, $P$ is a non-graded prime ideal, a contradiction. This implies that $\mg$ satisfies the Condition~$(K)$, and finishes the proof.
\end{proof}

For any ultragraph $\mg$, we denote by $\mathcal{E}_{c}(\mg)$ the set of all exclusive cycles in $\mg$. For any commutative unit ring $R$, we denote by $Spec(R)$ the prime spectrum of $R$. Theorem \ref{prime6} and Corollary \ref{prime7} provide us with a method of constructing non-graded prime ideals of ultragraph Leavitt path algebras.

%{\color{red} We have not defined $Spec(K[x, x^{-1}])$ anywhere. Should we say it is the prime spectrum?}

\begin{cor}\label{prime8}
Let $K$ be a field and $\mg$ an ultragraph. Then, the map $P\longmapsto (c, f(x))$, as indicated in Theorem~\ref{prime6}, defines a bijection between non-graded prime ideals of $L_K(\mg)$ and the set $\mathcal{E}_{c}(\mg)\times (Spec(K[x, x^{-1}])\setminus \{0\})$, where cycles obtained by permuting the vertices of a cycle are considered equal.
\end{cor}
\begin{proof}
Let $c\in \mathcal{E}_{c}(\mg)$ with $v = s(c)$. Let $\mathcal{H} =\{A\in \mg^0\mid w\ngeq v \text{ for all } w\in A\}$. As was shown in the proof of Corollary \ref{prime7}, $\mathcal{H}$ is a hereditary and saturated subset of $\mg^0$. It is also obvious that $G^0\setminus \bigcup \mathcal{H} = M(v)$. Therefore, for each irreducible polynomial $f(x)\in K[x, x^{-1}]$, by Theorem \ref{prime6} (iii), the ideal $P$, generated by 	$\{p_A\mid A\in \mathcal{H}\} \cup \{ p^{\mathcal{H}}_v\mid v\in  B_{\mathcal{H}}\} \cup \{f(s_c)\}$, is a non-graded prime ideal of $L_K(\mg)$. By Theorem \ref{prime6}, $P$ is uniquely determined by the cycle $c$ and the polynomial $f(x)$.
From these observations and Theorem \ref{prime6}, we obtain the statement, thus finishing the proof.
\end{proof}	

In \cite[Theorem 4.7]{gr:saccfulpavpsgrt} Royer and the second author characterized simplicity of the Leavitt path algebra associated with an ultragraph via partial skew group ring theory. We close this section with a simpler proof,  that uses Theorem~\ref{prime6}, of the simplicity criteria given in \cite[Theorem 4.7]{gr:saccfulpavpsgrt}.

\begin{cor}[{cf. \cite[Theorem 4.7]{gr:saccfulpavpsgrt}}]\label{simplicity}
Let $\mg$ an ultragraph and $K$ a field. Then, $L_K(\mg)$ is simple if, and only if, the following conditions hold:
\begin{itemize} 	
\item[(1)] The only hereditary and saturated subsets of $\mg^0$ are $\varnothing$ and $\mg^0$;
\item[(2)] Every cycle in $\mg$ has an exit.
\end{itemize}
\end{cor}	
\begin{proof}
Assume that $L_K(\mg)$ is simple. Then, $L_K(\mg)$ is a prime ring, and so $\mg$ is downward directed, by Theorem \ref{prime1}. Since the ideal generated by a non-empty proper hereditary saturated subset is a nonzero proper ideal of $L_K(\mg)$, the simplicity of $L_K(\mg)$ yields that the only hereditary and saturated subsets of $\mg^0$ are $\varnothing$ and $\mg^0$. Suppose that there exists a cycle $c$ without exits in $\mg$. Then, since $\mg$ is downward directed, by Theorem~\ref{prime6} (iii), there are infinitely many non-graded prime ideals of $L_K(\mg)$ generated by $f(s_c)$, where $f(x)$ is an irreducible polynomial in $K[x, x^{-1}]$, a contradiction. Therefore, every cycle in $\mg$ has an exit. 

Conversely, assume that $\mg$ satisfies the two conditions. Let $I$ be an ideal properly contained in $L_K(\mg)$. By condition (1), $\mathcal{H} = \{A\in \mg^0\mid p_A\in I\} = \varnothing$. If $I$ is nonzero then, by Lemma~\ref{prime2}, $I$ is generated by elements of the form $x = p_v + \sum^n_{i=1}k_is^{r_i}_c$, where $c$ is a unique cycle without exit based at $v$, and  $k_i\in K$ with at least one $k_i\neq 0$. But this contradicts the hypothesis that every cycle in $\mg$ has an exit. Thus, $I = 0$ and hence $L_K(\mg)$ is simple, as desired.
\end{proof}

%%%%%%%%%%%%%%%%%%%%%%%%%%%%%%%%%%

\section{Primitive ideals of ultragraph Leavitt path algebras}
The main aims of this section are to provide a criterion for primitivity of ultragraph Leavitt path algebras and to give a complete characterization of the primitive ideals of an ultragraph Leavitt path algebra (Theorem \ref{primitive3}).

In \cite{abr:opnpvnra}, Abrams, Bell and Rangaswamy gave a criterion for the Leavitt path algebra of an arbitrary graph to be primitive. In the following theorem, we extend this result to ultragraph Leavitt path algebras (the proof we present is based on the proof of \cite[Theorem 3.5]{abr:opnpvnra}, with the necessary modifications).

\begin{thm}\label{primitive1}
Let $K$ be an ultragraph and $K$ a field. Then $L_K(\mg)$ is primitive if, and only if, the following conditions hold:
\begin{itemize} 	
\item[(1)] $\mg$ is downward directed;
\item[(2)] Every cycle in $\mg$ has an exit.
\end{itemize}
\end{thm}
\begin{proof}
$(\Longrightarrow)$ Assume that $L_K(\mg)$ is primitive. Then, $L_K(\mg)$ is prime and so, by Theorem~\ref{prime1}, $\mg$ is downward directed. If there is a cycle $c$, based at a vertex
$v$ in $\mg$ with no exits, then $p_vL_K(\mg)p_v\cong K[x, x^{-1}]$
(via the isomorphism defined by: $p_{v}\longmapsto 1$, $s_{c}\longmapsto x$ and $s^*_{c}\longmapsto x^{-1}$) and so $p_vL_K(\mg)p_v$  is not primitive. On the other hand, since a nonzero corner of a primitive ring must again be primitive, $p_vL_K(\mg)p_v$ is primitive, a contradiction. Therefore, every cycle in $\mg$ has an exit.

Assume that $\mg$ satisfies the two conditions. Since $\mg$ is downward directed, by Theorem~\ref{prime1}, $L_K(\mg)$ is prime. Using \cite[Lemma~3.1]{abr:opnpvnra}, we may embed $L_K(\mg)$ as an ideal in a prime $K$-algebra $L_K(\mg)_1$. Let $v$ be an arbitrary vertex in $\mg$, and let $T(v) = \{w\in G^0\mid v \ge w\}$. Since $G^0$ is countable, $T(v)$ is at most countable, and so we may label the elements of $T(v)$ as $\{v_1, v_2, \hdots\}$. We inductively
define a sequence $\alpha_1, \alpha_2, \hdots$ of paths in $\mg$ and a sequence $w_1, w_2, \hdots$ of vertices in $\mg$ such that, for each $i\in \mathbb{N}$, 
\begin{itemize} 	
\item[(i)] $\alpha_j=\alpha_ip_j$ for some path $p_j$ with $w_i = s(p_j)$ whenever $i\le j$, and
\item[(ii)] $v_i\ge w_i$.
\end{itemize}
To do so, define $\alpha_1 = v_1 = w_1$. Now suppose $\alpha_1, \hdots , \alpha_n$ and $w_1, \hdots , w_n$ have been defined with the indicated properties for some $n\in \mathbb{N}$. Since $\mg$ is downward directed, there is a vertex $w_{n+1}$ such that $v_{n+1} \ge w_{n+1}$ and $w_n \ge w_{n+1}$. Let $p_{n+1}$ be a path in $\mg$ with $w_n  = s(p_{n+1})$ and $w_{n+1}\in r(p_{n+1})$, and define $\alpha_{n+1}= \alpha_{n}p_{n+1}$. Then, $\alpha_{n+1}$ is clearly seen to have the desired properties. We should note that $s_{\alpha_i}p_{w_i}s_{\alpha_i}^*s_{\alpha_j}p_{w_j}s_{\alpha_j}^* = s_{\alpha_j}p_{w_j}s_{\alpha_j}^*$ for all pair of positive integers $i\le j$. Moreover, $s_{\alpha_i}p_{w_i}s_{\alpha_i}^*\neq 0$ for all $i$.

We claim that every nonzero ideal $I$ of $L_K(\mg)_1$ contains some $s_{\alpha_i}p_{w_i}s_{\alpha_i}^*$. Since $L_K(\mg)$ is a nonzero ideal of the prime $K$-algebra $L_K(\mg)_1$, $I\cap L_K(\mg)$ is also a nonzero ideal of $L_K(\mg)$. Since every cycle in $\mg$ has an exit, repeating the method described in the proof of the direction $(\Longleftarrow)$ of Theorem \ref{cgutheo}, $I$ contains $p_w$ for some $w\in G^0$. Since  $\mg$ is downward directed, there exists $u\in G^0$ such that $v \ge u$ and $w\ge u$. But $v\ge u$ yields that $u = v_n$ for some $n\in \mathbb{N}$, and so $w\ge v_n$.
By the above construction, we have $v_n\ge w_n\in r(\alpha_n)$, and so there is a path $q$ in $\mg$ for which $w = s(q)$ and $w_n\in r(q)$. This implies that $p_{w_n} =p_{w_n} p_{r(q)}= p_{w_n}s^*_{p}p_w s_p\in I$, and so $s_{\alpha_n}p_{w_n}s_{\alpha_n}^*\in I$ and the claim is proved.

By \cite[Proposition 3.4]{abr:opnpvnra}, $L_K(\mg)_1$ is primitive, and so $L_K(\mg)$ is primitive, by \cite[Lemma 3.1]{abr:opnpvnra}, thus finishing the proof.
\end{proof}

In \cite[Theorem 4.3]{Ran:ttopiolpaoag} Rangaswamy characterized the primitive ideals of the Leavitt path algebra of an arbitrary graph. In \cite[Theorem 5.7]{Larki:piapioua19} Larki described  primitive gauge invariant ideals of ultragraph $C^*$-algebras. In the remainder of this section we extend Rangaswamy's result to ultragraph Leavitt path algebras. Before doing so, we recall useful notions of ultragraphs, introduced in \cite[Definition 5.6]{Larki:piapioua19}. Let $\mathcal{H}$ be a saturated hereditary subset of an ultragraph $\mg^0$. We say that a path $\alpha = \alpha_1 \hdots \alpha_n$ lies in $\mg\setminus \mathcal{H}$ whenever $r_{\mg}(\alpha)\in \mg^0\setminus \mathcal{H}$. We also say that $\alpha$ has {\it an exit in} $\mg\setminus \mathcal{H}$ if either $r_{\mg}(\alpha_i)\setminus \{s_{\mg}(\alpha_{i+1})\}\in \mg^0\setminus \mathcal{H}$ for some $i$, or there exists an edge $f$ such that $r_{\mg}(f)\in \mg^0\setminus \mathcal{H}$, $s_{\mg}(f) = s_{\mg}(\alpha_i)$ and $f\neq \alpha_i$.  We have the following simple fact.

\begin{lem}\label{primitive2}
Let $\mg$ be an ultragraph and $\mathcal{H}$ a saturated hereditary subset of $\mg^0$. Then, every cycle in $\mg/(\mathcal{H}, B_{\mathcal{H}})$ has an exit if, and only if, every cycle in $\mg\setminus \mathcal{H}$ has an exit in $\mg\setminus \mathcal{H}$.
\end{lem}
\begin{proof}
($\Longrightarrow$)	Assume that every cycle in $\mg/(\mathcal{H}, B_{\mathcal{H}})$ has an exit, and $\alpha=\alpha_1 \hdots \alpha_n$ is a cycle in $\mg\setminus \mathcal{H}$. Then, $r_{\mg}(\alpha)\in \mg^0\setminus \mathcal{H}$. Since $\mathcal{H}$ is hereditary, $\{s_{\mg}(\alpha_i)\}$ and $r_{\mg}(\alpha_i)$ lie in $\mg^0\setminus \mathcal{H}$ for all $1\le i\le n$, and so $\alpha_1 \hdots \alpha_n$ is a cycle in $\mg/(\mathcal{H}, B_{\mathcal{H}})$. By the hypothesis, it has an exit in $\mg/(\mathcal{H}, B_{\mathcal{H}})$, that is, there exists $1\le i\le n$ such that $r(\alpha_i)$ contains some sink $[v]$ in $\mg/(\mathcal{H}, B_{\mathcal{H}})$, or there exist $1\le i\le n$ and an edge $f$ in $\mg/(\mathcal{H}, B_{\mathcal{H}})$ with $s(f)\in r(\alpha_i)$ but $f\neq \alpha_{i+1}$. In the first case, we have that $\{w\}\notin \mathcal{H}$ and $w\in r_{\mg}(\alpha_i)\setminus \{s_{\mg}(\alpha_{i+1})\}$,  and so $r_{\mg}(\alpha_i)\setminus \{s_{\mg}(\alpha_{i+1})\}\in \mg^0\setminus \mathcal{H}$ (since $\mathcal{H}$ is hereditary). In the second case, we have $r_{\mg}(f)\in \mg^0\setminus \mathcal{H}$, $s_{\mg}(f) = s_{\mg}(\alpha_i)$ and $f\neq \alpha_i$. Therefore, $\alpha$ has an exit in $\mg/(\mathcal{H}, B_{\mathcal{H}})$.

$(\Longleftarrow)$ Let $\alpha =\alpha_1 \hdots \alpha_n$ be a cycle in $\mg/(\mathcal{H}, B_{\mathcal{H}})$. Then, $\alpha$ is also a cycle in $\mg\setminus \mathcal{H}$, and so $\alpha$ has an exit in $\mg\setminus \mathcal{H}$. This implies that $\alpha$ has an exit in $\mg/(\mathcal{H}, B_{\mathcal{H}})$, and the proof is finished.
\end{proof}

We are now in a position to provide the main result of this section.

\begin{thm}\label{primitive3}
Let $K$ be a field, $\mg$ an ultragraph, and $P$ an ideal of $L_K(\mg)$ with $\mathcal{H} = \{A\in \mg^0\mid p_A\in P\}$. Then, $P$ is a primitive ideal of $L_K(\mg)$ if, and only if, $P$ satisfies one of the following conditions: 
\begin{itemize}
\item[(i)] $P$ is a non-graded prime ideal;
\item[(ii)] $P$ is generated by the set $\{p_A\mid A\in \mathcal{H}\} \cup \{ p^{\mathcal{H}}_v\mid v\in B_{\mathcal{H}}\setminus\{w\}\}$ for some $w\in B_{\mathcal{H}}$ and $G^0\setminus \bigcup \mathcal{H} = M(w)$;		
\item[(iii)] $P$ is generated by the set $\{p_A\mid A\in \mathcal{H}\} \cup \{ p^{\mathcal{H}}_v\mid v\in  B_{\mathcal{H}}\}$, $G^0\setminus \bigcup \mathcal{H}$ is downward directed, and every cycle in $\mg\setminus \mathcal{H}$ has an exit in $\mg\setminus \mathcal{H}$.

\end{itemize}
\end{thm}	
\begin{proof}
($\Longrightarrow$)	It immediately follows from Theorem \ref{primitive1}, Lemma \ref{primitive2} and the cases (i), (ii) of Theorem \ref{prime6}.

($\Longleftarrow$) (i) Suppose $P$ is a non-graded prime ideal. By Theorem \ref{prime6} (iii), $P$ is generated by the set $\{p_A\mid A\in \mathcal{H}\} \cup \{ p^{\mathcal{H}}_v\mid v\in  B_{\mathcal{H}}\} \cup \{f(s_c)\}$, where $c$ is a cycle in $\mg$ based at a vertex $v$ such that no vertex on $c$ is the base of
another distinct cycle in $\mg$, $G^0\setminus \bigcup \mathcal{H} = M(v)$ and $f(x)$ is an irreducible polynomial in $K[x, x^{-1}]$. Then, the quotient ultragraph $\mg/(\mathcal{H}, B_{\mathcal{H}})$ is downward directed and contains the cycle $c$ without exits. By \cite[Theorem~3.4~ (1)]{ima:tlpaou}, there exists a $K$-algebra graded isomorphism $\phi: L_K(\mg)/I_{(\mathcal{H}, B_{\mathcal{H}})} \longrightarrow L_K(\mg/(\mathcal{H}, B_{\mathcal{H}}))$. Let $Q := \phi(P/I_{(\mathcal{H}, B_{\mathcal{H}})})$. Since $P$ is generated by the set $\{p_A\mid A\in \mathcal{H}\} \cup \{ p^{\mathcal{H}}_v\mid v\in  B_{\mathcal{H}}\} \cup \{f(s_c)\}$, $Q$ is generated by $f(s_c)$. Since $f(x)$ is irreducible in $K[x, x^{-1}]$, and $p_{[v]}L_K(\mg/(\mathcal{H}, B_{\mathcal{H}}))p_{[v]}\cong K[x, x^{-1}]$ via the isomorphism $\lambda$ defined by: $p_{[v]}\longmapsto 1$, $s_{c}\longmapsto x$ and $s^*_{c}\longmapsto x^{-1}$, we obtain that the ideal $p_{[v]}Q p_{[v]}$, being generated by $p_{[v]} f(s_c) p_{[v]} = f(s_c) = \lambda^{-1}(f(x))$, is a maximal ideal of the $K$-algebra $p_{[v]}L_K(\mg/(\mathcal{H}, B_{\mathcal{H}}))p_{[v]}$.
Now $p_{[v]}L_K(\mg/(\mathcal{H}, B_{\mathcal{H}}))p_{[v]}/p_{[v]}Q p_{[v]}\cong (p_{[v]}+ Q)(L_K(\mg/(\mathcal{H}, B_{\mathcal{H}}))/Q)(p_{[v]}+ Q)\cong (p_v + P)(L_K(\mg)/P)(p_v + P)$
under the natural isomorphisms. Since $p_{[v]}L_K(\mg/(\mathcal{H}, B_{\mathcal{H}}))p_{[v]}/p_{[v]}Q p_{[v]}$ is a field, $(p_v + P)(L_K(\mg)/P)(p_v + P)$ is also a field and, in particular, $(p_v + P)(L_K(\mg)/P)(p_v + P)$ is a commutative primitive ring. It is well known (see \cite[Theorem~1]{lrs:otpopr}) that a not-necessarily unital ring $R$ is primitive if, and only if, there is an idempotent $a\in R$ such that $aRa$ is a  primitive ring. These observations show that $L_K(\mg)/P$ is a primitive ring, and so $P$ is a primitive ideal.

(ii) Suppose $P$ is generated by the set $\{p_A\mid A\in \mathcal{H}\} \cup \{ p^{\mathcal{H}}_v\mid v\in B_{\mathcal{H}}\setminus\{w\}\}$ for some $w\in B_{\mathcal{H}}$ and $G^0\setminus \bigcup \mathcal{H} = M(w)$. Since $G^0\setminus \bigcup \mathcal{H} = M(w)$, $v\ge w$ for all $v\in G^0\setminus \bigcup \mathcal{H}$, and so in the ultragraph $\mg/(\mathcal{H}, B_{\mathcal{H}}\setminus\{w\})$ we have that $[v]\ge [w']$ for all $v\in G^0\setminus \bigcup \mathcal{H}$ and $[w']$ is a sink. This implies that $\mg/(\mathcal{H}, B_{\mathcal{H}}\setminus\{w\})$ is 
downward directed and every cycle in $\mg/(\mathcal{H}, B_{\mathcal{H}}\setminus\{w\})$ has an exit (since every vertex $[v]$ on any cycle in the ultragraph $\mg/(\mathcal{H}, B_{\mathcal{H}}\setminus\{w\})$ satisfies $[v]\ge [w']$). By Theorem~\ref{primitive1}, we have that $L_K(\mg/(\mathcal{H}, B_{\mathcal{H}}\setminus\{w\}))$ is a primitive ring. By \cite[Theorem 3.4 (1)]{ima:tlpaou}, $L_K(\mg)/P \cong L_K(\mg/(\mathcal{H}, B_{\mathcal{H}}\setminus\{w\}))$, and so  $L_K(\mg)/P$ is a primitive ring, that is, $P$ is a primitive ideal of $L_K(\mg)$.

(iii) Suppose that $P$ is generated by the set $\{p_A\mid A\in \mathcal{H}\} \cup \{ p^{\mathcal{H}}_v\mid v\in  B_{\mathcal{H}}\}$, $G^0\setminus \bigcup \mathcal{H}$ is downward directed, and every cycle in $\mg\setminus \mathcal{H}$ has an exit in $\mg\setminus \mathcal{H}$. By \cite[Theorem 3.4 (1)]{ima:tlpaou}, we have that $L_K(\mg)/P \cong L_K( \mg/(\mathcal{H},  B_{\mathcal{H}}))$.
Since $G^0\setminus \bigcup \mathcal{H}$ is downward directed, by the definition of the ultragraph $\mg/(\mathcal{H},  B_{\mathcal{H}})$, $\mg/(\mathcal{H},  B_{\mathcal{H}})$ is downward directed. Since every cycle in $\mg\setminus \mathcal{H}$ has an exit in $\mg\setminus \mathcal{H}$, and by Lemma \ref{primitive2}, every cycle in the ultragraph $\mg/(\mathcal{H},  B_{\mathcal{H}})$ has an exit. Now, by Theorem~ \ref{primitive1}, $L_K(\mg/(\mathcal{H},  B_{\mathcal{H}}))$ is a primitive ring, and so $P$ is a primitive ideal of $L_K(\mg)$, and the proof is finished.
\end{proof}

%%%%%%%%%%%%%%%%%%%%%%%%%%%%%%%%%%
\section{Chen simple modules of ultragraph Leavitt path algebras}
In \cite{c:irolpa}, Chen constructed simple modules for the Leavitt path algebra $L_K(E)$ of an arbitrary graph $E$, using sinks and the equivalence class of infinite paths tail-equivalent to a fixed infinite path in $E$, and their twisted modules. Chen's construction was extended by Ara and Rangaswamy in \cite{ar:fpsmolpa}, where the authors introduced additive classes of non-isomorphic simple $L_K(E)$-modules which were associated respectively to both infinite emitters $v$ and pairs $(c, f)$ consisting of exclusive cycles $c$ together with irreducible polynomials $f \in K[x, x^{-1}]\setminus\{1-x\}$. They call all these simple modules Chen simple modules. \'{A}nh and the third author, in \cite{anhnam2021}, gave another way to describe Chen simple modules and compute their annihilators. Royer and the second author, in \cite{gr:iaproulpa}, extended Chen's construction of simple modules for graph Leavitt path algebras to ultragraph Leavitt path algebras. The aim of this section is to extend Ara and Ragaswamy's construction of Chen simple modules for graph Leavitt path algebras to ulragraph Leavitt path algebras and compute their annihilators (Theorem \ref{Chenmod4}).
%Rangaswamy \cite{Rang:osmolpa} constructed an additive class of simple $L_K(E)$-modules, $N_{v\infty}$, by using infinite emitters $v$. By a different method from those presented in \cite{ar:fpsmolpa}, \'{Anh} and the second author \cite{anhnam} constructed simple $L_K(E)$-modules $S^f_c$ associated to pairs of consisting $(c, f)$ of simple closed paths $c$ together with irreducible polynomials $f$ in $K[x]$.

We begin this section by recalling useful notations of ultragraphs. Let $\mg$ be an arbitrary ultragraph. For $v\in G^0$, we define $$\mathcal{H}(v) := \{A\in \mathcal{G}^0\mid w\ngeq v \textnormal{ for all } w\in A\}.$$ For any infinite path $p \in \mg^{\infty}$, we define $$M(p) = \{v\in G^0\mid v \ge w \textnormal{ for some } w\in p^0\}$$ and
$$\mathcal{H}(p) = \{A\in \mathcal{G}^0\mid  w\ngeq v \textnormal{ for all } w\in A \textnormal{ and } v\in p^0\}.$$ Similar to what was done in the proof of Corollary \ref{prime7}, for any vertex $v$ which is either a sink or an infinite emitter, and for any infinite path $p$, the sets $\mathcal{H}(v)$ and $\mathcal{H}(p)$ are hereditary and saturated subsets of $\mathcal{G}^0$. 

Let $K$ be a field and $\mathcal{G}$ an ultragraph. Let $v$ be a sink in $\mathcal{G}$. In \cite[Proposition 3.9]{gr:iaproulpa} Royer and the second author defined the simple left $L_K(\mathcal{G})$-module $N_{v}$ to be the $K$-vector space having $$[v] :=\{p\in \mathcal{G}^*\mid |p|\ge 1, \, v \in r(p)\} \cup \{v\}$$ as a basis and with the scalar multiplication satisfying the following: for all $A\in \mathcal{G}^{0}$, $e\in \mathcal{G}^{1}$ and $\alpha\in [v]$,
\begin{align*}p_A\cdot \alpha= \begin{cases} \alpha &\textnormal{if } s(\alpha) \in A,\\  0 &\textnormal{otherwise,}\end{cases}\quad\quad
s_e\cdot \alpha= \begin{cases} e\alpha &\textnormal{if } s(\alpha) \in r(e),\\  0 &\textnormal{otherwise}\end{cases}\end{align*}
and 
\begin{align*}s^*_e\cdot \alpha= \begin{cases} \beta &\textnormal{if } \alpha =e\beta \textnormal{ for some } \beta\in [v] \textnormal{ with } |\beta| \ge 1,\\ v &\textnormal{if } \alpha = e,\\ 0 &\textnormal{otherwise.}\end{cases}\end{align*}

The following result extends \cite[Lemma 3.1]{ar:fpsmolpa} to ultragraph Leavitt path algeras.

\begin{lem}\label{Chenmod1}
Let $K$ be a field, $\mathcal{G}$ an ultragraph and $v$ a sink in $\mathcal{G}$. Then, the annihilator of the simple left $L_K(\mathcal{G})$-module $N_{v}$ is the ideal of $L_K(\mg)$ generated by $\{p_A\mid A\in \mathcal{H}(v)\} \cup \{ p^{\mathcal{H}(v)}_w\mid w\in  B_{\mathcal{H}(v)}\}$.
\end{lem}
\begin{proof}
We denote by $J$ and $I_{(\mathcal{H}(v), B_{\mathcal{H}(v)})}$ the annihilator of $N_{v}$ and the ideal of $L_K(\mathcal{G})$ generated by $\{p_A\mid A\in \mathcal{H}(v)\} \cup \{ p^{\mathcal{H}(v)}_w\mid w\in  B_{\mathcal{H}(v)}\}$, respectively. We first claim that $I_{(\mathcal{H}(v), B_{\mathcal{H}(v)})} \subseteq J$. Indeed, if $A \in \mathcal{H}(v)$, then it is obvious that $p_A\cdot \alpha = 0$ for all $\alpha \in [v]$. If $w\in B_{\mathcal{H}(v)}$ then, since $w \neq v$, we have that $p^{\mathcal{H}(v)}_w\cdot v = (p_w-\ds\sum_{e\in s^{-1}(w),\ r(e)\notin \mathcal{H}(v)}s_es^*_e)\cdot v =0$. Let $\alpha$ be a path of positive length in $\mathcal{G}^*$ such that $v \in r(\alpha)$. Let $f$ be the initial edge of $\alpha$, say $\alpha = f\beta$. Since $v\in r(\alpha)$, $u \ge v$ for some $u\in r(f)$, and so $r(f)\notin \mathcal{H}(v)$. This implies that $$p^{\mathcal{H}(v)}_w\cdot \alpha = (p_w-\ds\sum_{e\in s^{-1}(w),\ r(e)\notin \mathcal{H}(v)}s_es^*_e)\cdot f\beta =(f -f) \beta =0,$$ showing the claim. In order to show the reverse inclusion, we can consider $N_v$ as a simple left $L_K(\mg/(\mathcal{H}(v),  B_{\mathcal{H}(v)}))$-module. By \cite[Theorem 3.4 (1)]{ima:tlpaou}, $L_K(\mg)/I_{(\mathcal{H}(v), B_{\mathcal{H}(v)})} \cong L_K( \mg/(\mathcal{H},  B_{\mathcal{H}}))$. We next prove that $N_v$ is a faithful left $L_K(\mg/(\mathcal{H}(v),  B_{\mathcal{H}(v)}))$-module. Indeed, we denote by $\overline{J}$ the annihilator of the $L_K(\mg/(\mathcal{H}(v),  B_{\mathcal{H}(v)}))$-module $N_v$. Since every vertex in  $\mg/(\mathcal{H}(v),  B_{\mathcal{H}(v)})$ connects to $[v]$, we must have both $p_{[w]}\notin \overline{J}$ for all vertex $[w]$ in $\mg/(\mathcal{H}(v),  B_{\mathcal{H}(v)})$ and every cycle in  $\mg/(\mathcal{H}(v),  B_{\mathcal{H}(v)})$ has an exit. Then, by Corollary \ref{gi-cor} (3), every nonzero ideal of $L_K(\mg/(\mathcal{H}(v),  B_{\mathcal{H}(v)}))$ contains an element $p_{[w]}$ for some vertex $[w]$ in $\mg/(\mathcal{H}(v),  B_{\mathcal{H}(v)})$. Consequently, we obtain that $\overline{J} =0$, and the proof is finished.
\end{proof}

Let $\mg$ be an arbitrary ultragraph. For $p:= e_1\cdots e_n\cdots\in \mg^{\infty}$  and $n\ge 1$, we define $\tau_{> n}(p) = e_{n+1}e_{n+2}\cdots$. Two infinite paths $p, q$ are said to be \textit{tail-equivalent} (written $p\sim q$) if there exist positive integers $m, n$ such that $\tau_{> n}(p) = \tau_{> m}(q)$. Clearly $\sim$ is an equivalence relation on $\mg^{\infty}$, and we let $[p]$ denote the $\sim$ equivalence class of the infinite path $p$.

For $p:= e_1\cdots e_n\cdots\in \mg^{\infty}$, in \cite[Proposition 3.9]{gr:iaproulpa} Royer and the second author defined the simple left $L_K(\mathcal{G})$-module $V_{[p]}$ to be the $K$-vector space having $[p]$ as a basis and with the scalar multiplication satisfying the following:
for all $A\in \mathcal{G}^{0}$, $e\in \mathcal{G}^{1}$ and $\alpha\in [p]$,
\begin{align*}p_A\cdot \alpha= \begin{cases} \alpha &\textnormal{if } s(\alpha) \in A,\\  0 &\textnormal{otherwise,}\end{cases}\quad\quad
s_e\cdot \alpha= \begin{cases} e\alpha &\textnormal{if } s(\alpha) \in r(e),\\  0 &\textnormal{otherwise}\end{cases}\end{align*}
and 
\begin{align*}s^*_e\cdot \alpha= \begin{cases} \tau_{> 1}(\alpha) &\textnormal{if } \alpha =e\tau_{> 1}(\alpha),\\ 0 &\textnormal{otherwise.}\end{cases}\end{align*}

Now, motivated by Chen's construction \cite[Section 6]{c:irolpa} and Ara - Ragaswamy \cite[Section 3]{ar:fpsmolpa}, we consider twisted modules of the module $V_{[p]}$.  Let $c = e_1e_2\cdots e_n$ be a cycle in $\mg$ based at $v$. Then the path $c c c\cdots$ is an infinite path in $\mg$, which we denote by $c^{\infty}$. For $k\in K\setminus \{0\}$, by the universal property of ultragraph Leavitt path algebras and Theorem \ref{gutheo}, there is an algebra automorphism $\sigma_k: L_K(\mg) \longrightarrow L_K(\mg)$ such that $\sigma_k(p_A) = p_A$ for all $A\in \mg^0$, $\sigma_k(s_e) = s_e$ and $\sigma_k(s^*_e) = s^*_e$ for all $e\in \mg^1$ with $e\neq e_1$, and $\sigma_k(s_{e_1}) = ks_{e_1}$ and $\sigma_k(s^*_{e_1}) = k^{-1}s^*_{e_1}$. Then we have the simple left $L_K(\mg)$-module $V^{k}_{[c^{\infty}]}$, which is the twisted module $V^{\sigma_k}_{[c^{\infty}]}$. Denote by $*$ the scalar multiplication in 
$V^{k}_{[c^{\infty}]}$, we have $s_c \ast c^{\infty} = \sigma_k(s_c) c^{\infty}= ks_c c^{\infty} = kc^{\infty}$ and $s_c^* \ast c^{\infty} = k^{-1}c^{\infty}$.

Let $f(x)$ be an irreducible polynomial in $K[x, x^{-1}]$, and denote by $F$ the field $K[x, x^{-1}]/(f(x))$. Since $\overline{x}$ is invertible in $F$, $V^{\overline{x}}_{[c^{\infty}]}$ is a simple left $L_F(\mg)$-module. We denote by $V^f_{[c^{\infty}]}$ the $L_K(\mg)$-module obtained by restricting scalars from $L_F(\mg)$ to $L_K(\mg)$. Repeating the same method described in the proof of \cite[Lemma 3.1]{ar:fpsmolpa}, we obtain that
$V^f_{[c^{\infty}]}$ is a simple left $L_K(\mg)$-module.

\begin{lem}\label{Chenmod2}
Let $K$ be a field, $\mg$ an ultragraph, $c$ an exclusive cycle in $\mg$ based at a vertex $v$, and $f(x)$ an irreducible polynomial in $K[x, x^{-1}]$. Then the following statements hold:

$(1)$ $V^f_{[c^{\infty}]} \cong V_{[c^{\infty}]} \bigotimes_{K[x, x^{-1}]} K[x, x^{-1}]/(f(x))$ as left $L_K(\mg)$-modules;

$(2)$ The annihilator of $V^f_{[c^{\infty}]}$ is the ideal of $L_K(\mg)$ generated by $\{p_A\mid A\in \mathcal{H}(v)\} \cup \{ p^{\mathcal{H}(v)}_w\mid w\in  B_{\mathcal{H}(v)}\} \cup \{f(s_c)\}$.
\end{lem}
\begin{proof}
(1) Let $J$ be the subalgebra of $L_K(\mg)$ generated by $p_v$, $s_c$ and $s^*_c$. By the $\mathbb{Z}$-grading on $L_K(\mg)$, $J$ is isomorphic to the Jacobson algebra $K\langle x, y\mid yx = 1\rangle$ via the map: $p_v\longmapsto 1$, $s_e\longmapsto x$ and $s^*_c\longmapsto y$. Consider $V_{[c^{\infty}]}$ as the left $J$-module obtained by restricting scalars from $L_K(\mg)$ to $J$. By our hypothesis of the cycle $c$, we have $(p_v - s_c s^*_c)\alpha = 0$ for all $\alpha \in [c^{\infty}]$, and so  $p_v - s_c s^*_c$ annihilates $V_{[c^{\infty}]}$. This implies that $V_{[c^{\infty}]}$ is also a module over $J/(p_v - s_c s^*_c)$. We note that $J/(p_v - s_c s^*_c)$ is isomorphic to the Laurent polynomial algebra $K[x, x^{-1}]$ via the map: $p_v\longmapsto 1$, $s_e\longmapsto x$ and $s^*_c\longmapsto x^{-1}$. Therefore, $V_{[c^{\infty}]}$ is a module over $K[x, x^{-1}]$. It is, now, straightforward to show that $V^f_{[c^{\infty}]}$ is isomorphic to $V_{[c^{\infty}]} \bigotimes_{K[x, x^{-1}]} K[x, x^{-1}]/(f(x))$ as left $L_K(\mg)$-modules via the map: $\sum^n_{i=1}k_i\alpha_i\longmapsto \sum^n_{i=1}k_i\otimes \alpha_i$, where $k_i \in K[x, x^{-1}]/(f(x))$ and $\alpha_i\in [c^{\infty}]$.

(2) Proceeding as in the proof of Lemma \ref{Chenmod1}, we arrive at a simple left $L_K(\mg/(\mathcal{H}(v), B_{\mathcal{H}(v)}))$-module, and $\mg/(\mathcal{H}(v), B_{\mathcal{H}(v)})$ has a unique cycle without exits, which is $c$. Furthermore, the annihilator $\overline{J}$ of the $L_K(\mg/(\mathcal{H}(v),  B_{\mathcal{H}(v)}))$-module $V^f_{[c^{\infty}]}$ does not contain $p_{[w]}$ for all vertex $[w]$ in $\mg/(\mathcal{H}(v),  B_{\mathcal{H}(v)})$. Since $\overline{J}$ is a primitive ideal of $L_K(\mg/(\mathcal{H}(v), B_{\mathcal{H}(v)}))$, and by
Theorem \ref{primitive3}, there exists an irreducible polynomial $g$ in $K[x, x^{- 1}]$ such that $\overline{J}$ is the ideal generated by $g(s_c)$.
On the other hand, we have $$f(s_c) \ast c^{\infty} = \sigma_{\overline{x}}(f(s_c))c^{\infty} = f(\sigma_{\overline{x}}(s_c)) c^{\infty} = f(\overline{x})c^{\infty} = 0,$$ so $f(s_c)$ annihilates $V^f_{[c^{\infty}]}$. This shows that $f(x) = g(x)$, thus finishing the proof.
\end{proof}

Let $K$ be a field, $\mg$ an ultragraph and $v$ an infinite emitter in $\mg$. Let $N_{v\infty}$ be the $K$-vertor space having $$[v] :=\{p\in \mathcal{G}^*\mid |p|\ge 1, \, v \in r(p)\} \cup \{v\}$$ as a basis. We define, for each $A\in \mg^0$ and each $e\in \mg^1$, linear maps $P_A, S_e$ and $S^*_e$ on $N_{v\infty}$ as follows: for all $\alpha\in [v]$,
\begin{align*}P_A(\alpha) = \begin{cases} \alpha &\textnormal{if } s(\alpha) \in A,\\  0 &\textnormal{otherwise,}\end{cases}\quad
S_e(\alpha) = \begin{cases} e\alpha &\textnormal{if } s(\alpha) \in r(e),\\  0 &\textnormal{otherwise}\end{cases}\end{align*}
and 
\begin{align*}S^*_e(\alpha) = \begin{cases} \beta &\textnormal{if } \alpha =e\beta \textnormal{ for some } \beta\in [v] \textnormal{ with } |\beta| \ge 1,\\ v &\textnormal{if } \alpha = e,\\ 0 &\textnormal{otherwise.}\end{cases}\end{align*}
Then, similar to what was done in the proof of \cite[Proposition 3.9]{gr:iaproulpa}, we may show that the endomorphisms $\{P_A, S_e, S^*_e\mid A\in \mg^0,\, e\in \mg^1\}$ satisfy the relations analogous to (1) - (4)  in Definition~\ref{utraLevittpathalg}, and so there always exists a $K$-algebra homomorphism $\varphi: L_K(\mathcal{G})\longrightarrow End_K(N_{v\infty})$ given by $\varphi(p_A) = P_A$, $\varphi(s_e) = S_e$ and $\varphi(s^*_e) = S^*_e$. This implies that $N_{v\infty}$ may be made a left $L_K(\mg)$-module via the homomorphism $\varphi$. We denote the scalar multiplication on $N_{v\infty}$ by $\cdot$.

\begin{lem}\label{Chenmod3}
Let $K$ be a field, $\mathcal{G}$ an ultragraph and $v$ an infinite emitter in $\mathcal{G}$. Then, the following statements hold:

$(1)$ $N_{v\infty}$ is a simple left $L_K(\mg)$-module;

$(2)$ If $v \in B_{\mathcal{H}(v)}$, then the annihilator of $N_{v\infty}$ is the ideal of $L_K(\mg)$ generated by $\{p_A\mid A\in \mathcal{H}(v)\} \cup \{ p^{\mathcal{H}(v)}_w\mid w\in B_{\mathcal{H}(v)}\setminus\{v\}\}$;

$(3)$ If $r(s^{- 1}(v))\subseteq \mathcal{H}(v)$, then the annihilator of $N_{v\infty}$ is the ideal of $L_K(\mg)$ generated by $\{p_A\mid A\in \mathcal{H}(v)\} \cup \{ p^{\mathcal{H}(v)}_w\mid w\in B_{\mathcal{H}(v)}\}$.
\end{lem}
\begin{proof}
(1) It is completely similar to the proof of \cite[Proposition 3.9]{gr:iaproulpa}. We prove (2) and (3) below.

 %Let $P := I_{(\mathcal{H}(v),  B_{\mathcal{H}(v)}\setminus\{v\})}$. We then have that $P$ is a primitive ideal of $L_K(\mg)$, by Theorem\ref{prime6}.
Let $J := Ann_{L_K(\mg)}(N_{v\infty})$.  For $A\in \mathcal{H}(v)$, we have $w\ngeq v$ for all $w\in A$, and so $p_A\cdot \alpha =0$ for all $\alpha \in [v]$. This shows that $p_A\in J$ and $\{A\in \mg^0\mid p_A\in J\}  = \mathcal{H}(v)$. Assume that  $v \in B_{\mathcal{H}(v)}$. Let $w\in  B_{\mathcal{H}(v)}$ with $w\neq v$. Clearly $p^{\mathcal{H}(v)}_w \cdot\alpha = 0$ for all path $\alpha \in [v]$ with $w \neq s(\alpha)$. If $\alpha$ is a path in $[v]$ with $s(\alpha) = w$ and $\alpha = e \beta$, where $e\in \mg^1$, then
\[p^{\mathcal{H}(v)}_w \cdot\alpha = (p_w-\ds\sum_{f\in s^{-1}(w),\ r(f)\notin \mathcal{H}(v)}s_fs^*_f)\cdot e\beta =(e -e)\beta =0.\] This implies that $p^{\mathcal{H}(v)}_w \in J$. On the other hand, we have \[p^{\mathcal{H}(v)}_v\cdot v =  (p_w-\ds\sum_{f\in s^{-1}(w),\ r(f)\notin \mathcal{H}(v)}s_fs^*_f)\cdot v = v\cdot v - 0 = v \neq 0,\] so $I_{(\mathcal{H}(v), B_{\mathcal{H}(v)})\setminus \{v\})}\subseteq J$. We claim that $J$ is a graded ideal. Indeed, if $J$ is not graded, then, since $\{A\in \mg^0\mid p_A\in J\}  = \mathcal{H}(v)$ (and by Theorem \ref{primitive3} (1)), we must have $I_{(\mathcal{H}(v), B_{\mathcal{H}(v)})}\subseteq J$, and so $p^{\mathcal{H}(v)}_v \in J$, a contradiction, proving the claim. This shows that $J = I_{(\mathcal{H}(v), B_{\mathcal{H}(v)})\setminus \{v\})}$, proving (2).

Assume that $r(s^{- 1}(v))\subseteq \mathcal{H}(v)$. Then, repeating the same method described above, we have that $I_{(\mathcal{H}(v), B_{\mathcal{H}(v)})}\subseteq J$. If $J$ is not graded then, since $\{A\in \mg^0\mid p_A\in J\}  = \mathcal{H}(v)$ (and by Theorem \ref{primitive3} (1)), there exist an exclusive cycle $c$ in $\mg$ and an irreducible polynomial $f(x)$ in $K[x, x^{-1}]$ such that $J = I_{(\mathcal{H}(v), B_{\mathcal{H}(v)})} + (f(s_c))$ and $G^0\setminus \bigcup \mathcal{H}(v) = M(w)$, where $w = s(c)$ and $(f(s_c))$ is the ideal of $L_K(\mg)$ generated by $f(s_c)$. By the definition of $\mathcal{H}(v))$, we obtain that $v\in c^0 \subseteq M(v)$. In particular, there exists an edge $e\in \mg^1$ such that $s(e) = v$ and $r(e) \cap M(v) \neq \varnothing$, which contradicts our hypothesis that $r(s^{- 1}(v))\subseteq \mathcal{H}(v)$, and so $J$ is a graded ideal. Now, using Lemma \ref{prime4}, we immediately obtain that $J = I_{(\mathcal{H}(v), B_{\mathcal{H}(v)})}$, showing (3), and finishing the proof.
\end{proof}

Let $\mg$ be an arbitrary ultragraph and $K$ an arbitrary field. By a {\it Chen simple module} we mean a simple left $L_K(\mg)$-module of one of the following types:

(1) $N_v$, where $v$ is a sink in $\mg$;

(2) $N_{v\infty}$, where $v$ is an infinite emitter in $\mg$ such that $v \in B_{\mathcal{H}(v)}$;

(3) $N_{v\infty}$, where $v$ is an infinite emitter in $\mg$ such that $r(s^{- 1}(v))\subseteq \mathcal{H}(v)$;

(4) $V_{[p]}$, where $p$ is an infinite path in $\mg$;

(5) $V^f_{[c^{\infty}]}$, where $c$ is an exclusive cycle in $\mg$ and $f(x)$ is an irreducible polynomial in $K[x, x^{-1}]$ with
$f(x) \neq x - 1$.

We are now in a position to provide the main result of this section, which extends Ara and Rangaswamy's result \cite[Theorem 3.9]{ar:fpsmolpa} to ultragraph Leavitt path algebras.

\begin{thm}\label{Chenmod4}
Let $K$ be a field, $\mathcal{G}$ an ultragraph and $P$ a primitive ideal of $L_K(\mg)$. Then there exists a Chen simple $L_K(\mg)$-module $S$ such that the annihilator of $S$ is $P$.
\end{thm}
\begin{proof}
Let $\mathcal{H} = \{A\in \mg^0\mid p_A\in P\}$. By Theorem \ref{primitive3}, $P$ is exactly one of the following:
\begin{itemize}
\item[(i)] $P$ is a non-graded prime ideal;
\item[(ii)] $P$ is generated by the set $\{p_A\mid A\in \mathcal{H}\} \cup \{ p^{\mathcal{H}}_v\mid v\in B_{\mathcal{H}}\setminus\{w\}\}$ for some $w\in B_{\mathcal{H}}$ and $G^0\setminus \bigcup \mathcal{H} = M(w)$;		
\item[(iii)] $P$ is generated by the set $\{p_A\mid A\in \mathcal{H}\} \cup \{ p^{\mathcal{H}}_v\mid v\in  B_{\mathcal{H}}\}$, $G^0\setminus \bigcup \mathcal{H}$ is downward directed, and every cycle in $\mg\setminus \mathcal{H}$ has an exit in $\mg\setminus \mathcal{H}$.
\end{itemize}
Assume that (i) holds. Then, by Corollary \ref{prime8}, $P$ is generated by the set $\{p_A\mid A\in \mathcal{H}(v)\} \cup \{ p^{\mathcal{H}}_w\mid w\in  B_{\mathcal{H}(v)}\} \cup \{f(s_c)\}$, where $c$ is an exclusive cycle in $\mg$ based at a vertex $v$ and $f(x)$ is an irreducible polynomial in $K[x, x^{-1}]$. Applying Lemma \ref{Chenmod2} (2), we have that the annihilator of $V^f_{[c^{\infty}]}$ is exactly $P$. 

Next, assume that (ii) holds. We claim that $\mathcal{H} = \mathcal{H}(w)$. Since $G^0\setminus \bigcup \mathcal{H} = M(w)$, we immediately obtain that $\mathcal{H} \subseteq \mathcal{H}(w)$. Let $A\in \mg^0\setminus\mathcal{H}$, that is, $p_A \notin P$. Equivalently, $p_A$ is a nonzero element in the quotient $L_K(\mg)/P$. By \cite[Theorem~3.4~ (1)]{ima:tlpaou}, there exists a $K$-algebra graded isomorphism $\phi: L_K(\mg)/P\longrightarrow L_K(\mg/(\mathcal{H}, B_{\mathcal{H}\setminus\{w\}}))$ such that $\phi(p_A + P) = p_{[\overline{A}]}$. We then have $p_{[\overline{A}]}\neq 0$; equivalently, $[\overline{A}]$ must contain some vertex in $\mg/(\mathcal{H}, B_{\mathcal{H}\setminus\{w\}})$. Since the vertex set of $\mg/(\mathcal{H}, B_{\mathcal{H}\setminus\{w\}})$ is exactly the set $\{[v]\mid v\in G^0\setminus \bigcup \mathcal{H}\} \cup \{[w']\}$, either $A = \{w\}$ or $A$ contains a vertex $v\in G^0\setminus \bigcup \mathcal{H} = M(w)$. This implies that $A\notin \mathcal{H}(w)$, and so $\mathcal{H}(w) \subseteq \mathcal{H}$, proving the claim. Then, by Lemma \ref{Chenmod3} (2), $P$ is equal to the annihilator of the simple left $L_K(\mg)$-module $N_{w\infty}$.

Assume that (iii) holds. Then, $\mg/(\mathcal{H}, B_{\mathcal{H}})$ is downward directed and every cycle in $\mg/(\mathcal{H}, B_{\mathcal{H}})$ has an exit. If $\mg/(\mathcal{H}, B_{\mathcal{H}})$ has a sink $[v]$, then since $\mg/(\mathcal{H}, B_{\mathcal{H}})$ is downward directed, $[v]$ is a unique sink and $G^0\setminus \bigcup \mathcal{H} = M(v)$. Repeating the same method described above, we obtain that $\mathcal{H} = \mathcal{H}(v)$. There are then two cases: either $v$ is a sink in $\mg$, or $v$ is an infinite emitter in $\mg$ such that $r(s^{-1}(v)) \subseteq \mathcal{H}(v)$. By Lemmas \ref{Chenmod1} and \ref{Chenmod3} (3), the annihilator of $N_v$, respectively of $N_{v\infty}$, is exacly $P$.

Assume that $\mg/(\mathcal{H}, B_{\mathcal{H}})$ does not have any sink. Since $G^0\setminus \bigcup \mathcal{H}$ is countable, we may label the elements of $G^0\setminus \bigcup \mathcal{H}$ as $\{v_1, v_2, \hdots\}$. Repeating the same method described in the proof of Theorem \ref{primitive1}, there exist  a sequence $\alpha_1, \alpha_2, \hdots$ of paths in $\mg/(\mathcal{H}, B_{\mathcal{H}})$ and a sequence $[w_1], [w_2], \hdots$ of vertices in $\mg/(\mathcal{H}, B_{\mathcal{H}})$ such that, for each $i\in \mathbb{N}$, 
\begin{itemize} 	
\item[(i)] $\alpha_j=\alpha_ip_j$ for some path $p_j$ with $[w_i] = s(p_j)$ whenever $i\le j$, 
\item[(ii)] $|\alpha_i|\ge i$ for all $i$, and
\item[(iii)] $[v_i]\ge [w_i]$ for all $i$.
\end{itemize}
Now, we may use the paths $\alpha_i$ to construct an infinite path $\alpha$ such that each vertex of $G^0\setminus \bigcup \mathcal{H}$
connects to a vertex in $\alpha^0$. Since every cycle in $\mg/(\mathcal{H}, B_{\mathcal{H}})$ has an exit, $\alpha$ is not tail-equivalent to  $c^{\infty}$ for all exclusive cycle $c$. Since $G^0\setminus \bigcup \mathcal{H} = M(\alpha) = G^0\setminus \bigcup \mathcal{H}(\alpha)$, repeating the same method described in the proof of (ii) we obtain $\mathcal{H} = \mathcal{H}(\alpha)$. Now, similarly to what was done in the proof of Lemma \ref{Chenmod1}, we obtain that the annihilator of the simple left $L_K(\mg)$-module $V_{[\alpha]}$ is exactly $I_{(\mathcal{H}(\alpha), B_{\mathcal{H}(\alpha)})} = P$, and the proof is finished.
\end{proof} 

%Note that if $c$ and $d$ are closed paths in $E$ such that $c = d^n$, then $c^{\infty}=d^{\infty}$ as elements of $E^{\infty}$. The infinite path $p$ is called \textit{rational} in case $p\sim c^{\infty}$ for some closed path $c$. If $p\in E^{\infty}$ is not rational we say $p$ is \textit{irrational}. We denote by $E^{\infty}_{rat}$ and $E^{\infty}_{irr}$ the sets of rational and irrational paths in $E$, respectively.

%%%%%%%%%%%%%%%%%%%%%%%%%%%%%%%%%%

\section{Exel's Effros-Hahn conjecture for ultragraph Leavitt path algebras}
The aim of this section is to show that the Exel's Effros-Hahn conjecture holds for ultragraph Leavitt path algebras (Theorem \ref{Exel's Effros-Hahn}).

\subsection{Steinberg algebras and Exel's Effros-Hahn conjecture} Steinberg algebras were introduced in \cite{stein:agatdisa} in the context of discrete inverse semigroup algebras and independently in \cite{cfst:aggolpa} as a model for Leavitt path algebras. 

%We recall the notion of the Steinberg algebra as a universal algebra generated by certain compact open subsets of a Hausdorff ample groupoid. 

%{\color{red} Maybe we should use another letter for groupoid since we already use G for ultragraphs?}

A \textit{groupoid} is a small category in which every morphism is invertible. It can also be viewed as a generalization of a group that has a partial binary operation. 
%Let G be a groupoid. 
Let $\mathcal{G}$ be a groupoid. If $x \in \mathcal{G}$, $s(x) = x^{-1}x$ is the source of $x$ and $r(x) = xx^{-1}$ is its range. The pair $(x,y)$ is is composable if and only if $r(y) = s(x)$. The set $\mathcal{G}^{(0)}:= s(\mathcal{G}) = r(\mathcal{G})$ is called the \textit{unit space} of $\mathcal{G}$. Elements of $\mathcal{G}^{(0)}$ are units in the sense that $xs(x) = x$ and
$r(x)x = x$ for all $x\in \mathcal{G}$. For $U, V \subseteq \mathcal{G}$, we define
\begin{center}
	$UV = \{\alpha\beta \mid \alpha\in U, \beta\in V \text{\, and \,} r(\beta) = s(\alpha)\}$ and $U^{-1} = \{\alpha^{-1}\mid \alpha\in U\}$.
\end{center}
 The \textit{isotropy group} of a unit $x\in \mathcal G^{(0)}$ is the group $\mathcal G_x = \{g\in \mathcal G \mid s(g)=r(g)=x\}$.

A \textit{topological groupoid} is a groupoid endowed with a topology under which the inverse map is continuous, and such that the composition is continuous with respect to the relative topology on $\mathcal{G}^{(2)} := \{(\beta, \gamma)\in \mathcal{G}^2\mid s(\beta) = r(\gamma)\}$ inherited from $\mathcal{G}^2$. An \textit{\'{e}tale groupoid} is a topological groupoid $\mathcal{G},$ whose unit space  $\mathcal{G}^{(0)}$ is locally compact Hausdorff, and such that the domain map $s$ is a local homeomorphism. 
%In this case, the range map $r$ and the multiplication map are local homeomorphisms and  $\mathcal{G}^{(0)}$ is open in $\mathcal{G}$ \cite{r:egatq}. 

An \textit{open bisection} of $\mathcal{G}$ is an open subset $U\subseteq \mathcal{G}$ such that $s|_U$ and $r|_U$ are homeomorphisms onto an open subset of $\mathcal{G}^{(0)}$. Similar to \cite[Proposition 2.2.4]{p:gisatoa} we have that $UV$ and $U^{-1}$ are compact open bisections for all compact open bisections $U$ and $V$ of an  \'{e}tale groupoid  $\mathcal{G}$.
If in addition $\mathcal{G}$  is Hausdorff, then $U\cap V$ is a compact open bisection (see \cite[Proposition 3.7]{stein:agatdisa}). An \'{e}tale groupoid $\mathcal{G}$ is called \textit{ample} if  $\mathcal{G}$ has a base of compact open bisections for its topology.

Let $\mathcal{G}$ be an ample groupoid, and $K$ a field with the discrete topology. We denote by $K^{\mathcal{G}}$ the set of all functions from $\mathcal{G}$ to $K$. Canonically,  $K^{\mathcal{G}}$ has the structure of a $K$-vector space with
operations defined pointwise.

\begin{defn} 
Let $\mathcal{G}$ be an ample groupoid, and $K$ any field.   Let $A_K(\mathcal{G})$ be the $K$-vector subspace of $K^{\mathcal{G}}$ generated by the set
\begin{center}
$\{1_U\mid U$ is a compact open bisection of $\mathcal{G}\}$,
\end{center}
where $1_U: \mathcal{G}\longrightarrow K$ denotes the characteristic function on $U$. The \textit{multiplication} of $f, g\in A_K(\mathcal{G})$ is given by the convolution
\[(f\ast g)(\gamma)= \sum_{\gamma = \alpha\beta}f(\alpha)g(\beta)\] for all $\gamma\in \mathcal{G}$. The $K$-vector subspace $A_K(\mathcal{G})$, with convolution, is called the \textit{Steinberg algebra} of $\mathcal{G}$ over $K$.
\end{defn}
By \cite[Proposition 4.6]{stein:agatdisa}, $A_K(\mathcal{G})$ equipped with convolution is a $K$-algebra.
It is useful to note that $1_U\ast 1_V = 1_{UV}$ for compact open bisections $U$ and $V$. In particular, $1_U\ast 1_V = 1_{U\cap V}$
whenever $U$ and $V$ are compact open subsets of $\mathcal{G}^{(0)}$ (see \cite[Proposition 4.5]{stein:agatdisa}).

Let $K$ be a field, $\mathcal{G}$ a groupoid, and $u\in\mathcal{G}^{(0)}$. Define $L_u := s^{-1}(u)$. The isotropy group $\mathcal{G}_u$ acts on the right of $L_u$. Consider the $K$-vector space $KL_u$ with basis $L_u$. The right action of $\mathcal{G}_u$ on $L_u$ induces a free right $K\mathcal{G}_u$-module structure on $KL_u$ (see \cite[Proposition 7.7]{stein:agatdisa}). Moreover, by \cite[Proposition 7.8]{stein:agatdisa}, $KL_u$ is a left $A_K(\mathcal{G})$-module with the scalar multiplication defined by: \[f\cdot x = \sum_{y\in L_u}f(yx^{-1})y\] for all $f\in A_K(\mathcal{G})$ and $x\in L_u$. It is useful to note (see \cite[Proposition~7.8]{stein:agatdisa}) that
\begin{align*}1_U\cdot x= \begin{cases} yx &\textnormal{if there is a } y\in U \textnormal{ such that } s(y) = r(x),\\  0 &\textnormal{otherwise}.\end{cases}\end{align*}	For a left $K\mathcal{G}_u$-module $V$, we define the corresponding
\textit{induced} left $A_K(\mathcal{G})$-module to be $$\text{Ind}_u(V) = KL_u\otimes_{K\mathcal{G}_u}V.$$ In \cite[Proposition 7.19]{stein:agatdisa} Steinberg obtained the following interesting results.

\begin{thm}[{\cite[Proposition~7.19]{stein:agatdisa}}]\label{Induction}
Let $K$ be a field, $\mathcal{G}$ an ample groupoid, $u\in \mathcal{G}^{(0)}$, and  $V$  a simple left $K\mathcal{G}_u$-module. Then $\rm{Ind}$$_u(V)$ is a simple left $A_K(\mathcal{G})$-module. Moreover, if $V$ and $W$ are non-isomorphic simple left $K\mathcal{G}_u$-modules, then $\rm{Ind}$$_u(V)\ncong \rm{Ind}$$_u(W)$.

%$(2)$ $(${\cite[Proposition~7.20]{stein:agatdisa}}$)$ If $u$ and $v$ are elements in distinct orbits, then $\rm{Ind}$$_u(V)$ and $\rm{Ind}$$_v(V)$ are not isomorphic.
\end{thm}

The original Effros-Hahn conjecture \cite{Effros-Hahn1967a, Effros-Hahn1967b} suggested that every primitive ideal of a crossed product of an amenable locally compact group with a commutative $C^*$-algebra should be induced from a primitive ideal of an isotropy group. The result was proved by Sauvageot \cite{Sauvageot79} for discrete groups and a more general result than the original conjecture was proved by Gootman and Rosenberg in \cite{Gootnam-Rosenberg79}. Crossed products of the above form are special cases of groupoid $C^*$-algebras and analogues of the Effros-Hahn conjecture in the groupoid setting were achieved by Renault \cite{Renault91} and Ionescu-Williams \cite{Ionescu-Williams09}. R. Exel conjectured at the PARS meeting in Gramado, 2014  that an analogue of the Effros-Hahn conjecture should hold for Steinberg algebras. 

\begin{conj}[{Exel's Effros-Hahn conjecture}]
Let $K$ be a field, $\mg$ an ample groupoid and $I$ a primitive ideal of $L_K(\mg)$. Then $I =Ann_{A_K(\mathcal{G})}(\rm{Ind}$$_u(M))$ for some $u\in \mathcal{G}^{(0)}$ and simple left $K\mathcal{G}_u$-module $M$.
\end{conj}

We should mention results relating to this conjecture. In \cite{Demeneghi2019}, motivated by Dokuchaev and Exel's result \cite{Doku-Exel2017}, Demeneghi showed that every ideal of the Steinberg algebra of an ample groupoid is an intersection of kernels of induced representations from isotropy subgroups. In \cite{stein:ioegaaeehc} Steinberg proved that every primitive ideal of the Steinberg algebra over a commutative unital ring is the kernel of an induced representation from an isotropy group. Consequently, he obtained Exel's Effros-Hahn conjecture in two cases: namely, if the base ring is Artinian and all isotropy groups are finite; or if the base ring is the field of complex numbers and the isotropy groups are all locally finite abelian or finite.

\subsection{The groupoid associated to an ultragraph}

In \cite{dgv:uavlggwatgut}, a groupoid is associated with an arbitrary ultragraph in a manner that the corresponding Steinberg algebra is isomorphic to the Leavitt path algebra associated with the ultragraph. In this level of generality, as far as the author's knowledge goes, the use of labelled spaces and the associated concepts are unavoidable. To fully recall these concepts here would be a lengthy endeavor, which would repeat known results. We, therefore, try to keep the preliminaries to a minimum and refer the reader to Section~3.1 of \cite{dgv:uavlggwatgut} for the detailed explanation.

A \emph{labelled graph} consists of a graph $\mathcal E$ together with a surjective \emph{labelling map} $\mathcal L:\mathcal E^1\to\mathcal A$, where $\mathcal A$ is a fixed non-empty set, called an \emph{alphabet}, and whose elements are called \emph{letters}. $\mathcal A^{\ast}$ stands for the set of all finite \emph{words} over $\mathcal A$, together with the \emph{empty word} $\omega$, and $\mathcal A^{\infty}$ is the set of all infinite words over $\mathcal A$. %We consider $\mathcal A^{\ast}$ as a monoid with operation given by concatenation. In particular, given $\alpha\in \mathcal A^{\ast}\setminus\{\omega\}$ and $n\in\mathbb N^*$, $\alpha^n$ represents $\alpha$ concatenated $n$ times and $\alpha^{\infty}\in \mathcal A^{\infty}$ is $\alpha$ concatenated infinitely many times.
The labelling map $\mathcal L$ extends in the obvious way to $\mathcal L:\mathcal E^n\to\mathcal A^{\ast}$ and $\mathcal L:\mathcal E^{\infty}\to\mathcal A^{\infty}$. $\mathcal L^{n}=\mathcal L(\mathcal E^n)$ is the set of \emph{labelled paths $\alpha$ of length $|\alpha|=n$}, and $\mathcal L^{\infty}=\mathcal L({\mathcal E^\infty})$ is the set of \emph{infinite labelled paths}. We consider $\omega$ as a labelled path with $|\omega|=0$, and set $\mathcal L^{\geq 1}=\cup_{n\geq 1}\mathcal L^n$, $\mathcal L^* =\{\omega\}\cup\mathcal L^{\geq 1}$, and $\mathcal L^{\leq \infty} = \mathcal L^{*}\cup \mathcal L^{\infty}$. 
%Define $\overline{\mathcal L^{\infty}}=\{\alpha\in\mathcal A^{\infty}\mid \alpha_{1,n}\in\mathcal L^*,\, \forall\, n\in\mathbb N\}$, and write $\overline{\mathcal L^{\leq \infty}}=\mathcal L^*\cup\overline{\mathcal L^\infty}$.

Next, we recall the labelled space associated with an ultragraph.

\begin{defn}
Fix an ultragraph $\mathcal{G}=(G^0, \mathcal{G}^1, r,s)$. Let $\mathcal E_\mathcal{G}=(\mathcal E^0_\mathcal{G}, \mathcal E^1_\mathcal{G},r^{\prime},s^{\prime})$, where $\mathcal E^0_\mathcal{G}=G^0$, $\mathcal E^1_\mathcal{G}=\{(e,w):e\in\mathcal{G}^1, w\in r(e)\}$ and define  $r^{\prime}(e,w)=w$ and $s^{\prime}(e,w)=s(e)$. Set $\mathcal A :=\mathcal{G}^1$, $\mathcal B:=\mathcal G^0$, and define $\mathcal L_\mathcal{G}:\mathcal E^1_\mathcal{G} \to\mathcal A$ by $\mathcal L{}_\mathcal{G}(e,w)=e$. Then,
	$(\mathcal E_\mathcal{G},\mathcal L_\mathcal{G},\mathcal B)$ is the normal labelled space associated with $\mathcal G$.
\end{defn}

There is an inverse semigroup associated to the labelled space $(\mathcal E_\mathcal{G},\mathcal L_\mathcal{G},\mathcal B)$ above, whose semilattice of idempotents is 
\[E(S)=\{(\alpha, A, \alpha) \ | \ \alpha\in\mathcal L^* \ \mbox{and} \ A\in\mathcal B_{\alpha}\}\cup\{0\},\]
where, for $\alpha\in\mathcal L^*$,  \[\mathcal B_\alpha=\mathcal B\cap P({r(\alpha)})=\{A\in\mathcal B\mid A\subseteq r(\alpha)\}.\]

The unit space of the groupoid associated with the ulgragraph $\mathcal G$ is the set of all tight filters in $E(S)$, which we denote by $\mathbf T$. In \cite{dgv:uavlggwatgut}, elements of $\mathbf T$ are described in terms of labelled paths and the associated filters. To obtain this correspondence, first it is observed that for each infinite path in $\mathcal G$ there is only one filter in $E(S)$ associated with it, and for each finite path $\alpha$, there is a filter in $E(S)$ associated to each filter $\xi$ on $\mathcal B_\alpha$ (\cite[Remark~3.5]{dgv:uavlggwatgut}). Then, the following description of $\mathbf T$ is proved.

\begin{prop}[{\cite[Proposition~3.6]{dgv:uavlggwatgut}}]\label{tightfilters}
	For each infinite path $\alpha$ on $\mathcal G$, there is a unique element $\xi\in\mathbf T$, whose associated word is $\alpha$. If $\xi^\alpha$ is a filter of finite type, then $\xi^\alpha\in\mathbf T$ if and only if one of the following holds:
	\begin{enumerate}[(i)]
		\item There exists $v\in G^0_s$ such that $\xi_{|\alpha|}=\uparrow_{\mathcal B_\alpha}\{v\}$, where $\uparrow_{\mathcal B_\alpha}\{v\}=\{A \in \mathcal B_\alpha\mid v\in A\}$.
		\item For all $A\in\xi_{|\alpha|}$, $|\varepsilon(A)|=\infty$.
		\item For all $A\in\xi_{|\alpha|}$, $|A\cap G^0_s|=\infty$.
	\end{enumerate}
\end{prop}

Using the above, we have that the elements of $\mathbf T$ can be described in two ways:

\begin{itemize}
	\item if $\xi\in \mathbf T$ is of infinite type, then it is completely described by the infinite path associated to it;
	\item if $\xi\in \mathbf T$ is of finite type, then it is described by a pair $(\alpha,\mathcal F)$, where $\alpha$ is the labelled path associated to $\xi$ and $\mathcal F$ is a filter in $\mathcal B_\alpha$ satisfying one of the three conditions of Proposition~\ref{tightfilters}.
\end{itemize}

The groupoid associated with an ultragraph is a Deaconu-Renault groupoid. To define it, first we recall the definition of the shift map on $ \mathbf T$. 

For each $n\in\mathbb N$, we let $\mathbf T^{(n)}=\{\xi^\alpha \in\mathbf T\mid |\alpha|\geq n\}$. The shift map $\sigma:\mathbf T^{(1)}\to\mathbf T$ is defined as follows.
\begin{enumerate}
	\item If the associated path of $\xi$ is an infinite path $\alpha_1\alpha_2\alpha_3\ldots$, then $\sigma(\xi)$ is the filter associated to the path $\alpha_2\alpha_3\ldots$ as in Proposition~\ref{tightfilters}.
	
	\item If the associated path of $\xi$ is a finite path $\alpha_1\alpha_2\ldots\alpha_n$ with $n\geq 2$, then $\sigma(\xi)$ is the filter associated to the pair $(\alpha_2\ldots\alpha_n,\xi_n)$.
	
	\item If the associated path of $\xi$ is an edge $e\in\mathcal G^1$, then $\sigma(\xi)$ is the filter associated to the pair $(\omega,\uparrow_{\mathcal B} \xi_1)$, recalling that $\omega$ is the empty word.
\end{enumerate}

Finally, the groupoid associated with $\mathcal G$ is given by 
\[\Gamma(\mathbf T,\sigma)=\{(\xi,m-n,\eta)\in\mathbf T\times\mathbb{Z}\times\mathbf T\mid m,n,\in\mathbb N,\xi\in\mathbf T^{(n)},\eta\in\mathbf T^{(m)},\sigma^n(\xi)=\sigma^m(\eta)\},\]
with product and inverse given by
\begin{itemize}
	\item $(\xi,k,\eta)(\zeta,l,\rho)=(\xi,k+l,\rho)$ for $(\xi,k,\eta),(\zeta,l,\rho)\in\Gamma(\mathbf T,\sigma)$ such that $\eta=\zeta$, and 
	\item $(\xi,k,\eta)^{-1}=(\eta,-k,\xi)$ for $(\xi,k,\eta)\in\Gamma(\mathbf T,\sigma)$,
\end{itemize}
respectively. 

To specify a basis for $\Gamma(\mathbf T,\sigma)$, we first recall a basis for the topology on $\mathbf T$. 
For each $e\in E(S)$, define
\begin{equation*} \label{dfn:tight.spec.basic.open.neigh}
	V_e=\{\xi\in\mathbf T\mid e\in\xi\},
\end{equation*}
and for $\{e_1,\ldots,e_n\}$ a finite (possibly empty) set in $E(S)$, define
\[ V_{e:e_1,\ldots,e_n}= V_e\cap V_{e_1}^c\cap\cdots\cap V_{e_n}^c=\{\xi\in\mathbf T\mid e\in\xi,e_1\notin\xi,\ldots,e_n\notin\xi\}.\]
Let $E(S)^+=\bigcup_{n=1}^\infty E(S)^n$, where $E(S)^n$ is the Cartesian product of $n$ copies of $E(S)$. For each $n\in\mathbb N$ and $\mathbf{e}=(e,e_1,\ldots,e_n)\in E(S)^+$, define $V_{\mathbf{e}}=V_{e:e_1,\ldots,e_n}$. Then, the family $\{V_{\mathbf{e}}\}_{\mathbf{e}\in E(S)^+}$ is a basis for the topology on $\mathbf T$.

Now, a basis for the topology on $\Gamma(\mathbf T,\sigma)$  is given by the family of sets of the form
\begin{center}
$\mathcal{V}(U, V, m, n) = \{(\xi, m - n, \eta)\in \Gamma(\mathbf T,\sigma)\mid  (\xi, \eta)\in U\times V, \sigma^m(\xi) = \sigma^n(\eta)\},$
\end{center}
where $m, n\in \bb{N}$, $U$ is an open subset of $\mathbf T^{(m)}$ and $V$ is an open subset of $\mathbf T^{(n)}$. Another basis for the topology on $\Gamma(\mathbf T,\sigma)$ is given by the sets of the form
\begin{center}
$ Z_{(e,e_1,\ldots,e_k)}=\{ (\xi, m - n, \eta)\in \Gamma(\mathbf T,\sigma)\mid  \eta\in  V_{(e,e_1,\ldots,e_k)}, \sigma^m(\xi) = \sigma^n(\eta)\},$ 
\end{center}
see \cite[Section~4]{Boava_Gilles_Mortari}.

Given an ultragraph $\mathcal G$, it is proved in \cite{dgv:uavlggwatgut} that the Steinberg algebra associated with $\Gamma(\mathbf T,\sigma)$ is isomorphic to $L_K(\mathcal G)$. In fact, following the isomorphism described in \cite[Theorem~5.5]{dgv:uavlggwatgut}, we obtain that there is an isomorphism $$\phi_{\mg}:L_K(\mathcal G)\rightarrow A_K(\Gamma(\mathbf T,\sigma))$$ such that $\phi_{\mg}(p_A)=1_{Z_{(\omega,A,\omega)}}$ for all $A\in \mathcal G^0$, $\phi_{\mg} (s_e)=1_{Z_{(e,r(e),\omega)}}$ and $\phi_{\mg}(s_e^*)= 1_{Z_{(\omega,r(e),e)}}$, for all $e\in \mathcal G^1$.

Next, we prove that ultragraphs behave like graphs in terms of the isotropy groups of the associated groupoids. We have the following result, which generalizes \cite[Proposition~6]{Steinberg} to ultragraphs.

\begin{prop}\label{isotropygroup} Let $\mathcal G$ be an ultragraph, $\Gamma(\mathbf T,\sigma)$ the associated groupoid, and $\xi \in \mathbf T$. Then, the isotropy group $\Gamma(\mathbf T,\sigma)_\xi$ is trivial, unless $\xi$ is a filter of infinite type associated with $\xi = \rho \gamma \gamma \cdots$, where $\rho$ is a finite path of length greater than zero and $\gamma$ is a closed path, in which case $\Gamma(\mathbf T,\sigma)_\xi\cong \mathbb Z$.
\end{prop}
\begin{proof}
By \cite[Proposition~6.10]{Gilles_Danie}, a tight filter $\xi^\alpha\in \mathbf T$ has non-trivial isotropy if, and only if, there exists $\rho$ and $\gamma$, labelled paths of length greater than zero, such that $\alpha = \rho \gamma \gamma \cdots$. Let $\xi^\alpha\in \mathbf T$ be such that $\alpha = \rho \gamma \gamma \cdots$. Then, $\Gamma(\mathbf T,\sigma)_{\xi^\alpha} = \{(\xi^\alpha,m,\xi^\alpha)\mid m\in \mathbb Z\}$, since if $m, n\geq 0$ then $\alpha = \rho \gamma^n \gamma\gamma \cdots = \rho \gamma^m \gamma\gamma \cdots $ shows that $(\xi^\alpha,m-n,\xi^\alpha)\in \Gamma(\mathbf T,\sigma)_{\xi^\alpha}$, thus finishing the proof.
\end{proof}

Let $K$ be a field, $\mg$ an ultragraph and $p$ an infinite path in $\mg$. We then have $$L_{(\xi^{p}, 0, \xi^{p})} := s^{-1}_{\Gamma(\mathbf T,\sigma)}((\xi^{p}, 0, \xi^{p})) = \{(\xi^{q}, k, \xi^{p})\mid q\in [p],\, k\in \mathbb{Z}\}.$$
Consider the $K$-vector space $KL_{(\xi^{p}, 0, \xi^{p})}$ with basis $L_{(\xi^{p}, 0, \xi^{p})}$. Then $KL_{(\xi^{p}, 0, \xi^{p})}$ is a left $A_K(\Gamma(\mathbf T,\sigma))$-module with the  multiplication satisfying the following:
\begin{align*}1_{Z_{(\omega,A,\omega)}}\cdot (\xi^{q}, k, \xi^{p})= \begin{cases} (\xi^{q}, k, \xi^{p}) &\textnormal{if } s(q) \in A,\\  0 &\textnormal{otherwise}\end{cases}\end{align*}

\begin{align*}1_{Z_{(e, r(e),\omega)}}\cdot (\xi^{q}, k, \xi^{p})= \begin{cases} (\xi^{eq}, k+1, \xi^{p}) &\textnormal{if } s(q) \in r(e),\\  0 &\textnormal{otherwise}\end{cases}\end{align*}

\begin{align*}1_{Z_{(\omega, r(e), e)}}\cdot (\xi^{q}, k, \xi^{p})= \begin{cases} (\xi^{\tau_{>1}(q)}, k-1, \xi^{p}) &\textnormal{if } q = e\tau_{> 1}(q),\\  0 &\textnormal{otherwise}\end{cases}\end{align*}
for all $A\in \mathcal{G}^0$, $e\in \mathcal{G}^1$ and $(\xi^{q}, k, \xi^{p})\in L_{(\xi^{p}, 0, \xi^{p})}$. By Theorem \ref{Induction}, $\text{Ind}_{(\xi^{p}, 0, \xi^{p})}(M) = KL_{(\xi^{p}, 0, \xi^{p})}\otimes_{K\Gamma(\mathbf T,\sigma)_{(\xi^{p}, 0, \xi^{p})}}M $ is a simple left $A_K(\mathcal{G})$-module for all simple left $K\Gamma(\mathbf T,\sigma)_{(\xi^{p}, 0, \xi^{p})}$-module $M$.

Assume that $p$ is not tail-equivalent to $c^{\infty}$ for all closed path $c$ in $\mg$. We have that the isotropy group $\Gamma(\mathbf T,\sigma)_{(\xi^{p}, 0, \xi^{p})}$ is trivial by Proposition \ref{isotropygroup}, so $K\Gamma(\mathbf T,\sigma)_{(\xi^{p}, 0, \xi^{p})} \cong K$ and $K$ is a simple $K\Gamma(\mathbf T,\sigma)_{(\xi^{p}, 0, \xi^{p})}$-module. We then have that $$\text{Ind}_{(\xi^{p}, 0, \xi^{p})}(K) = KL_{(\xi^{p}, 0, \xi^{p})}\otimes_{K\Gamma(\mathbf T,\sigma)_{(\xi^{p}, 0, \xi^{p})}}K = KL_{(\xi^{p}, 0, \xi^{p})}\otimes_KK \cong KL_{(\xi^{p}, 0, \xi^{p})}$$ as $A_K(\Gamma(\mathbf T,\sigma))$-modules. By this, and the above isomorphism $\phi_{\mathcal{G}}$, $\text{Ind}_{(\xi^{p}, 0, \xi^{p})}(K)$ may be viewed as a simple left $L_K(\mg)$-module. It is clear that \[\text{Ind}_{(\xi^{p}, 0, \xi^{p})}(K)\cong V_{[p]}\]
as left $L_K(\mathcal{G})$-modules.

Assume that $p$ is tail-equivalent to $c^{\infty}$ for some closed path $c$ in $\mg$. By Proposition~\ref{isotropygroup}, we have that the isotropy group $\Gamma(\mathbf T,\sigma)_{(\xi^{p}, 0, \xi^{p})}= \{(\xi^p,m,\xi^p)\mid m\in \mathbb Z\}\cong \mathbb Z$ via the map: $(\xi^p,m,\xi^p)\longmapsto m$, and so $K\Gamma(\mathbf T,\sigma)_{(\xi^{p}, 0, \xi^{p})} \cong K[x, x^{- 1}]$ via the map: $(\xi^p,1,\xi^p)\longmapsto x$ and
$(\xi^p, - 1,\xi^p)\longmapsto x^{- 1}$. We note that every simple $K[x, x^{-1}]$-module is of the form $K[x, x^{-1}]/(f(x))$, where $(f(x))$ is the ideal of $K[x, x^{-1}]$ generated by an irreducible polynomial $f(x)$.

Let $f(x)$ be an irreducible polynomial in $K[x, x^{-1}]$. Then, $$\text{Ind}_{(\xi^{p}, 0, \xi^{p})}(K[x, x^{-1}]/(f(x)))= KL_{(\xi^{p}, 0, \xi^{p})}\otimes_{K[x, x^{-1}]}K[x, x^{-1}]/(f(x))$$ is a simple left $A_K(\Gamma(\mathbf T,\sigma))$-module (see Theorem~\ref{Induction}). By this note and the above isomorphism $\phi_{\mathcal{G}}$, $\text{Ind}_{(\xi^{p}, 0, \xi^{p})}(K[x, x^{-1}]/(f(x)))$ may be viewed as a simple left $L_K(\mg)$-module. If $c$ is an exclusive cycle then, by Lemma~\ref{Chenmod2} (1), we have that \[\text{Ind}_{(\xi^{p}, 0, \xi^{p})}(K[x, x^{-1}]/(f(x)))\cong V^f_{[p]}\] as left $L_K(\mg)$-modules.

Let $v$ be either a sink or an infinite emitter in $\mg$, and $\xi^v :=\ \uparrow_{\mathcal B} \{v\}$. Then, $$L_{(\xi^{v}, 0, \xi^{v})} := s^{-1}_{\Gamma(\mathbf T,\sigma)}((\xi^{v}, 0, \xi^{v})) = \{(\uparrow_{\mathcal B_{\alpha}} \{v\}, |\alpha|, \xi^{v})\mid \alpha\in \mg^*, v\in r(\alpha), |\alpha|\ge 1\}.$$
Consider the $K$-vector space $KL_{(\xi^{v}, 0, \xi^{v})}$ with basis $L_{(\xi^{v}, 0, \xi^{v})}$. Then, $KL_{(\xi^{v}, 0, \xi^{v})}$ is a left $A_K(\Gamma(\mathbf T,\sigma))$-module with the  multiplication satisfying the following:
\begin{align*}1_{Z_{(\omega,A,\omega)}}\cdot (\uparrow_{\mathcal B_{\alpha}} \{v\}, |\alpha|, \xi^{v})= \begin{cases} (\uparrow_{\mathcal B_{\alpha}} \{v\}, |\alpha|, \xi^{v}) &\textnormal{if } s(\alpha) \in A,\\  0 &\textnormal{otherwise}\end{cases}\end{align*}

\begin{align*}1_{Z_{(e, r(e),\omega)}}\cdot (\uparrow_{\mathcal B_{\alpha}} \{v\}, |\alpha|, \xi^{v})= \begin{cases} (\uparrow_{\mathcal B_{e\alpha}} \{v\}, |\alpha|+1, \xi^{v}) &\textnormal{if } s(q) \in r(e),\\  0 &\textnormal{otherwise}\end{cases}\end{align*}

\begin{align*}1_{Z_{(\omega, r(e), e)}}\cdot (\uparrow_{\mathcal B_{\alpha}} \{v\}, |\alpha|, \xi^{v})= \begin{cases} (\uparrow_{\mathcal B_{\beta}} \{v\}, |\alpha|-1, \xi^{v}) &\textnormal{if } |\alpha|\ge 2 \textnormal{ and }\alpha = e\beta,\\
(\xi^v,0, \xi^v) & \textnormal{if } \alpha = e,\\  0 &\textnormal{otherwise}\end{cases}\end{align*}
for all $A\in \mathcal{G}^0$, $e\in \mathcal{G}^1$ and $(\uparrow_{\mathcal B_{\alpha}} \{v\}, |\alpha|, \xi^{v})\in L_{(\xi^{v}, 0, \xi^{v})}$. By Theorem \ref{Induction}, $\text{Ind}_{(\xi^{v}, 0, \xi^{v})}(M) = KL_{(\xi^{v}, 0, \xi^{v})}\otimes_{K\Gamma(\mathbf T,\sigma)_{(\xi^{v}, 0, \xi^{v})}}M $ is a simple left $A_K(\mathcal{G})$-module for all simple left $K\Gamma(\mathbf T,\sigma)_{(\xi^{v}, 0, \xi^{v})}$-module $M$.
On the other hand, we have that the isotropy group $\Gamma(\mathbf T,\sigma)_{(\xi^{v}, 0, \xi^{v})}$ is trivial by Proposition \ref{isotropygroup}, so $K\Gamma(\mathbf T,\sigma)_{(\xi^{v}, 0, \xi^{v})} \cong K$ and $K$ is a simple $K\Gamma(\mathbf T,\sigma)_{(\xi^{v}, 0, \xi^{v})}$-module. We then have that $$\text{Ind}_{(\xi^{v}, 0, \xi^{v})}(K) = KL_{(\xi^{v}, 0, \xi^{v})}\otimes_{K\Gamma(\mathbf T,\sigma)_{(\xi^{v}, 0, \xi^{v})}}K = KL_{(\xi^{v}, 0, \xi^{v})}\otimes_KK \cong KL_{(\xi^{v}, 0, \xi^{v})}$$ as $A_K(\Gamma(\mathbf T,\sigma))$-modules. By this, and the above isomorphism $\phi_{\mathcal{G}}$, $\text{Ind}_{(\xi^{p}, 0, \xi^{p})}(K)$ may be viewed as a simple left $L_K(\mg)$-module. We also have that 
\begin{align*}\text{Ind}_{(\xi^{v}, 0, \xi^{v})}(K) \cong \begin{cases} N_v&\textnormal{if } v \textnormal{ is a sink},\\  N_{v\infty} &\textnormal{if } v \textnormal{ is an infinite emitter}\end{cases}\end{align*}
as left $L_K(\mathcal{G})$-modules.

From these observations and Theorems \ref{Chenmod4} and \ref{primitive3}, we immediately obtain the following result.

\begin{thm}\label{Exel's Effros-Hahn}
Let $K$ be a field and $\mg$ an ultragraph. Then, every Chen simple $L_K(\mg)$-module may be induced from a simple module over the group $K$-algebra of some isotropy group of the groupoid $\Gamma(\mathbf T,\sigma)$. Consequently, Exel's Effros-Hahn conjecture holds for ultragraph Leavitt path algebras.
\end{thm}

\section*{Acknowledgements}
Daniel Gon\c{c}alves was partially supported by Conselho Nacional de Desenvolvimento Cient\'{i}fico e Tecnol\'{o}gico - CNPq and Capes-PrInt, Brazil.

%%%%%%%%%%%%%%%%%%%%%%%%%%%%%%%%%%
%
%
%%%%     BIBLIOGRAPHY
%
%
%%%%%%%%%%%%%%%%%%%%%%%%%%%%%%%%%%

\vskip 0.5 cm \vskip 0.5cm {

\end{document}